\documentclass[smallextended,referee,envcountsect,]{svjour3}

\smartqed

\usepackage{latexsym,amssymb,amsmath}
\usepackage{amsxtra,amscd,enumerate,verbatim,amsfonts,amscd,pstricks}
\usepackage{graphicx}
\usepackage{framed}
\newtheorem{assumption}{Assumption}[section]
\usepackage[linkcolor=blue, urlcolor=blue, citecolor=blue,
colorlinks, bookmarks]{hyperref}
\journalname{JNVA}

\begin{document}

\title{Second-Order Strong Optimality and Second-Order Duality for Nonsmooth Constrained Multiobjective  Fractional Programming Problems}

\titlerunning{Second-order Strong Optimality and Duality for Nonsmooth Constrained MFP}

\author{Jiawei Chen  \and Luyu Liu \and Yibing Lv \and Debdas Ghosh \and Jen-Chih Yao}

\institute{
Jiawei Chen \and Luyu Liu
\at
School of Mathematics and Statistics, Southwest University,  Chongqing 400715, China\\
Emails:  J.W.Chen713@163.com (J. Chen),\,  lyliu124@163.com (L. Liu)
\and
Yibing Lv
\at
School of Information and Mathematics, Yangtze University, Jingzhou 434023, China. \\
Email:  Yibinglv@yangtzeu.edu.cn
\and
Debdas Ghosh
\at
Department of Mathematical Sciences, Indian Institute of Technology (BHU), Varanasi, Uttar Pradesh 221005, India \\
Email: debdas.mat@iitbhu.ac.in
\and
Jen-Chih Yao
\at
Research Center for Interneural Computing, China Medical University, Taichung 40447, Taiwan\\
Email: yaojc@mail.cmu.edu.tw
}

\date{Received: date / Accepted: date}

\maketitle

\begin{abstract}
This paper investigates constrained \emph{nonsmooth multiobjective fractional programming problem} (NMFP) in real Banach spaces. It derives a quotient calculus rule for computing the first- and second-order Clarke derivatives of fractional functions involving locally Lipschitz functions. A novel second-order Abadie-type regularity condition is presented, defined with the help of the Clarke directional derivative and the P\'ales-Zeidan second-order directional derivative. We establish both first- and second-order strong necessary optimality conditions,  which contain some new information on multipliers and imply the strong KKT necessary conditions, for a Borwein-type properly efficient solution of NMFP by utilizing generalized directional derivatives. Moreover, it derives second-order sufficient optimality conditions for NMFP under a second-order generalized convexity assumption. Additionally, we derive duality results between NMFP and its second-order dual problem under some appropriate conditions.
\end{abstract}


\section{Introduction}

\emph{Multiobjective fractional programming} (MFP) represents a crucial model in operations research, finding widespread applications in fields such as computer vision, portfolio optimization, management science, image processing, and communications. Various works have investigated and utilized MFP problems, including those in \cite{Bector1968,Borwein1976,Schaible1983,Luc,Jahn,David,Sunlongchai,ShenYuI,ShenYuII,AnsariKY}. Due to the nonsmooth nature of the functions involved in realistic MFP problems, practical challenges emerged in fields like economics, decision theory, optimal control, scheduling, machine learning, engineering, and game theory. As a result, MFPs have been extensively examined using subdifferentials and directional derivatives, as evident from the studies in \cite{Singh1986,Bector1993,Liu1996,Kim2006,Jayswal2015,Das2022}. Within the optimization theory of MFP, duality theory and optimality conditions constitute fundamental areas of study. Recent research has been attempting to identify both first- and second-order characterizations for efficient points of MFP problems, leveraging a range of first- and second-order directional derivatives.

First- and second-order characterizations for a local weak efficient solution of NMFP were derived in \cite{Khanh2015} by generalized derivatives as the first- and second-order approximations of the involving functions. Khanh and Tung \cite{Khanh2020} investigated the \emph{first-order} Karush-Kuhn-Tucker (KKT) conditions for a local Borwein-type proper efficient solution of nonsmooth semi-infinite multiobjective programming by a Mangasarian-Fromovitz regularity condition.
Su and Hang \cite{VanSu2022} applied Hadamard directional derivative to establish a first-order quotient rule and derived first-order conditions for local weak efficient solutions of NMFP. They also derived, in \cite{VanSu2023}, second-order conditions for strict local efficient solutions of MFP by utilizing second-order P\'ales-Zeidan-type upper directional derivative. The duality results and first-order sufficient optimality conditions for single-objective fractional programming problems were obtained under the ($F,\alpha,\rho,d$)-convexity assumption in \cite{Liang2001}. Further, in \cite{Liang2003}, the results proposed in  \cite{Liang2001} were extended for smooth multiobjective fractional programming problems under a convexity assumption. In \cite{Yuan2006}, a notion of ($C,\alpha,\rho,d$)-convexity was introduced and applied to investigate the first-order sufficient optimality conditions and duality for nonsmooth minimax fractional programming problems. The first-order optimality conditions and duality for nondifferentiable MFP problems were also explored in \cite{Chinchuluun2007}. Very recently, in \cite{Pokharna2023}, duality results and optimality conditions for $E$-minimax fractional programming were investigated assuming the $E$-invexity of the involved functions, and applied to multiobjective optimization.

\begin{framed}
To the best of the authors' knowledge, only very few results exist concerning the second-order duality and second-order strong necessary optimality conditions for NMFPs. This study focuses specifically on the realm of Borwein-type proper efficient solutions for NMFP problems, which are known to be more potent than (weak) efficient solutions. By leveraging the Clarke directional derivative and second-order P\'ales-Zeidan generalized directional derivative, we aim to explore second-order characterizations of Borwein-type proper efficient solutions and report second-order Mond-Weir duality theory for NMFP problems involving both equality and inequality constraints. Notably, our investigation considers the locally Lipschitz continuous functions rather than those that are Fr\'echet differentiable.
\end{framed}

Towards deriving second-order duality results and strong optimality conditions, we propose a \emph{generalized second-order Abadie-type regularity condition} (GSOARC) with the help of P\'ales-Zeidan second-order generalized directional derivatives, which do not require the continuous Fr\'echet differentiability and existence of second-order directional derivatives of all involved functions that were a requirement in \cite{Feng2019}. The first- and second-order strong KKT optimality conditions for Borwein proper efficient solutions of NMFP problems are reported under mild conditions. We show that the GSOARC cannot be relaxed to the generalized second-order Guignard-type regularity condition in obtaining the derived second-order strong optimality conditions. Furthermore, sufficient optimality conditions for NMFP are derived under a generalized second-order convexity assumption. Additionally, the Mond-Weir-type second-order duality results for NMFP are derived. The weak, strong, and converse duality results between NMFP and its second-order dual problem are identified under some suitable conditions.

The rest of the paper is demonstrated as follows. In Section \ref{sec:2}, we recollect some basic notions and present quotient calculus for locally Lipschitz functions. We introduce a GSOARC and investigate the first- and second-order strong KKT conditions for a Borwein-type properly efficient solution of NMFP in Section \ref{sec:3}. Subsequently, second-order sufficient optimality conditions are presented in Section \ref{sec:4}  under some second-order generalized convexity assumptions. In Section \ref{sec:5}, the second-order Mond-Weir-type dual problem is introduced. We also study duality results between NMFP  and the corresponding second-order dual problem. Finally, we give the conclusions in Section \ref{sec:6}.

\section{Preliminaries}\label{sec:2}
Let $Z$ be a real reflexive Banach space with the norm $\|\cdot\|$, and $U$ be a nonempty open subset of $Z$.
 For $A \subseteq Z$, its closure, topological interior, and convex hull are denoted by $\mbox{cl}\, A$, $\text{int}\, A$, and $\mbox{co}\, A$, respectively.
 Let $Z^*$ be the topological dual space of $Z$,  and $\langle \cdot, \cdot \rangle$ denote the coupling between $Z^*$ and $Z$.
 We denote the nonnegative orthant and positive orthant of  $\mathbb{R}^n$ by $\mathbb{R}^{n}_{+}$ and $\mathbb{R}^n_{++}$, respectively.
 Let
$$ F:=(F_{1},F_{2},\ldots,F_{p})^{\top}: Z  \rightarrow \mathbb{R}^{p}, $$
$$ f:=(f_{1},f_{2},\ldots,f_{p})^{\top}: Z  \rightarrow \mathbb{R}^{p}, $$
$$ g:=(g_{1},g_{2},\ldots,g_{m})^{\top}: Z  \rightarrow \mathbb{R}^{m} $$
 $$\,\text{and}\,
  h:=(h_{1},h_{2},\ldots,h_{l})^{\top}: Z  \rightarrow \mathbb{R}^{l}$$
be vector-valued functions, and let $F(x) \in \mathbb{R}^{p}_{++}$ and $\frac{f(x)}{F(x)} :=\left( \frac{f_1 (x)}{F_1 (x)},\frac{f_2 (x)}{F_2 (x)},\ldots, \frac{f_p (x)}{F_p (x)}\right)^{\top}$ for all $x\in U$,  where
the superscript $\top$ denotes the transpose.
For $x=(x_1,x_2,\ldots,x_p)^{\top}$ and $y=(y_1,y_2,\ldots,y_p)^{\top}\in \mathbb{R}^p$, we undertake the following conventional notations:
\begin{align*}
&x=y \,\,\,\,\,\,\,\,\,~ \Longleftrightarrow\,\, x_i=y_i,\, i=1,2,\ldots,p,\\
&x \leqq y \,\,\,\,\,\, \Longleftrightarrow\,\, x - y \in - \mathbb{R}^{p}_{+},\\
&x\le y\,\,\,\,\,~ \Longleftrightarrow\,\, x- y \in -\mathbb{R}^{p}_{+}\setminus \{0\}, \\
&x < y \,\,\,\,\,\, \Longleftrightarrow\,\, x - y \in - \mathbb{R}^{p}_{++},\\
&z = x * y \,  \Longleftrightarrow\,\, z = (x_1 y_1, x_2 y_2, \ldots, x_p y_p)^\top \\
\text{ and } & w = x^\top y \, \Longleftrightarrow\,\, w = x_1y_1 + x_2 y_2 + \cdots + x_p y_p.
\end{align*}
The relations $y \geqq x$, $y \ge x$, and $y > x$ simply mean $x \leqq y$, $x \le y$ and $x< y$, respectively.
For brevity, we use the following notations:
\[I:=\{1,2,\ldots ,p\},~ J:=\{1,2,\ldots ,m\},~ K:=\{1,2,\ldots ,l\}\] and  for each $x\in U$, we denote
\begin{align*}
& \left( \frac{f}{F} \right) (x) :=\frac{f(x)}{F(x)},\,\, \left( \frac{f_i}{F_i} \right) (x) :=\frac{f_i( x)}{F_i( x)},\,\,\, i\in I \\
\text{ and } & \left( \frac{1}{F} \right) (x) :=\left( \frac{1}{F_1(x)}, \frac{1}{F_2(x)}, \ldots ,\frac{1}{F_p(x)} \right)^{\top}.
\end{align*}

\begin{framed}
\noindent
In this study, we focus on the \emph{first- and second-order Borwein proper optimality} and \emph{second-order Mond-Weir duality} of the  following nonsmooth multiobjective fractional programming problem:
\begin{equation}\label{mfp:1}
\begin{aligned}
\mbox{(NMFP)}\quad
\left\{ \begin{array}{lll}
\min\,\, & \frac{f(x)}{F(x)} \\
\mbox{s.t.}\, &
g(x)  \leqq 0,\,\, h( x)  =0,\, x\in U.
\end{array} \right.
\end{aligned}
\end{equation}
The set of all feasible points of NMFP is denoted  by $X$, i.e.,
 \begin{align}\label{def_of_X}
   X:=\left\{ x\in U\,:\, g(x) \leqq 0,\, h( x ) =0 \right\}.
   \end{align}
We also consider the following parametric problem associated with NMFP:
\begin{equation}\label{smfp:1}
\begin{aligned}
\mbox{($s$-MFP)}\quad
\left\{ \begin{array}{lll}
&\min & \, (f-s\ast F)(x)\\
&\mbox{s.t.}& \,\,
	 g( x)  \leqq 0,\, h( x)  =0,\, x\in U,
\end{array} \right.
\end{aligned}
\end{equation}
where $s:=(s_1,s_2,\ldots,s_p)\in \mathbb{R}^p$ and
$\left( f-s\ast F \right) ( x ) =f(x) -s\ast F( x )$.
\end{framed}

\noindent
We next recall some standard definitions \cite{VanSu2022} and well-known results, which will be useful later.

\begin{definition}\label{def:6}
An element $x_0\in X$ is said to be
\begin{enumerate}[(i)]
\item a \emph{Pareto efficient} solution of NMFP iff, there is no $x\in X$ such that
  $$\left( \frac{f}{F} \right) ( x ) \le \left( \frac{f}{F} \right) ( x_0 ).$$

\item  \emph{weak efficient solution} of NMFP iff, there is no $x\in X$ such that
$$\left( \frac{f}{F} \right) ( x ) < \left( \frac{f}{F} \right) ( x_0 ).$$
\end{enumerate}

Similarly, one can define (weak) Pareto efficient solutions of $s$-MFP by substituting $(f-s\ast F)$ for $\left( \frac{f}{F} \right)$.
\end{definition}

The following result shows an equivalence between the (weak) efficient solutions of NMFP and $s$-MFP.
\begin{lemma}\label{lem:2.1}
A feasible solution $x_0\in X$ is a (weak) Pareto efficient solution of {\rm NMFP} if and only if it is a (weak) Pareto efficient solution of {\rm $s$-MFP}, where $s :=\frac{f(x_0)}{F(x_0)}$.
\end{lemma}

\begin{proof}
A proof for $x_0$ being a weak Pareto point is given in \cite[Proposition 1]{VanSu2022}. For a proof of $x_0$ being a Pareto efficient point, we observe for any $x\in U$ with $x \neq x_0$ that
\begin{align*}
&\left( \frac{f}{F} \right) \left( x \right) -  \left( \frac{f}{F} \right) \left( x_0 \right) \not\in - \mathbb{R}^{p}_{+}\setminus\{0\}\\
\Longleftrightarrow \,&\left( \frac{f}{F} \right) \left( x \right) - s \not\in - \mathbb{R}^{p}_{+}\setminus\{0\} \\
\Longleftrightarrow\, & \frac{1}{F(x)}(f(x) - s \ast F(x)) \not\in - \mathbb{R}^{p}_{+}\setminus\{0\}\\
\Longleftrightarrow \,&f(x)- s \ast F(x) \not\in - \mathbb{R}^{p}_{+} \setminus\{0\} \text{ because } \mathbb{R}^{p}_{+} \text{ is a cone and } F(x) \in \mathbb{R}^p_{++} \\
\Longleftrightarrow \, &(f- s \ast F)(x) \not\le (f- s \ast F)(x_0).
\end{align*}
Hence, the result follows.  The proof is completed.
\end{proof}

\begin{definition}\label{def:1}
\cite{Clarke1983}
A function $\vartheta: U\rightarrow \mathbb{R}$ is said to be   \emph{G\^ateaux differentiable} at $x_0\in U$ in a direction $v_0 \in Z$ iff,
 the limit
\begin{equation}\label{gds:1}
\lim_{t\rightarrow 0^+}\frac{\vartheta( x_0 + tv_0 ) - \vartheta(x_0)}{t}
\end{equation}
exists. If the limit exists, it is denoted by $\vartheta'(x_0; v_0)$ and called the G\^ateaux derivative of $\vartheta$ at $x_0$ along the direction $v_0$.

Let $x_0 \in U$. If $\vartheta'(x_0; v)$ exists for any $v \in Z$ and there is a continuous linear function $\vartheta'_{G}( x_0): Z \rightarrow \mathbb{R}$ such that
\[
\vartheta'(x_0; v) =  \langle \vartheta'_{G}(x_0), v \rangle ~\text { for each } v\in Z,
\]
then we say that $\vartheta'_{G}(x_0)$ is the G\^ateaux derivative of $\vartheta$ at $x_0$.

A vector-valued function $\vartheta:=\left(\vartheta_1, \vartheta_2, \ldots, \vartheta_p \right)^{\top}\,:\, U\rightarrow \mathbb{R}^p$ is called G\^ateaux differentiable at $x_0\in U$ if its components $\vartheta_i, i\in I$, are G\^ateaux differentiable at $x_0$.
\end{definition}

\begin{remark}
\begin{enumerate}[(i)]
\item
In some instances, as discussed in \cite{Jahn}, the limit described in \eqref{gds:1} can be alternatively expressed with $t\rightarrow 0$. However, for consistency with other definitions of directional derivatives, we use the one-sided limit in this paper.

\item
It cannot be generally concluded that $\vartheta$ is continuous at $x$ even if $\vartheta$ is G\^ateaux differentiable at $x$.

\item
If $\vartheta$ is Fr\'echet differentiable at $x_0$, then $\vartheta$ is G\^ateaux differentiable at $x_0$ and $\nabla \vartheta( x_0 ) =\vartheta'_{G}( x_0)$, where $\nabla \vartheta( x_0)$ is the Fr\'echet derivative of $\vartheta$ at $x_0$.
\end{enumerate}
\end{remark}

\begin{definition}\label{def:3}\cite{Clarke1983}
Let $\vartheta: U\rightarrow \mathbb{R}$ be a real-valued function and $x_0 \in U$.
\begin{enumerate}[(i)]
\item
The (Clarke) \emph{generalized directional derivative} of $\vartheta$ at $x_0 \in U$ in the direction $v_0\in Z$ is defined by
\begin{align*}
\vartheta^{\circ}(x_0; v_0) := \limsup_{y\rightarrow x_0,~ t\rightarrow 0^+}\frac{\vartheta( y_0 + tv_0 ) - \vartheta( y_0 )}{t}.
\end{align*}

\item $\vartheta$ is said to be \emph{regular} in the sense of Clarke at $x_0$ if  $\vartheta'( x_0; v )$ exists and $\vartheta'( x_0; v ) = \vartheta^{\circ}( x_0; v )$ holds for all $v\in Z$. \\
\end{enumerate}
\end{definition}

For any given $x_0\in X$, with regard to the problem \eqref{mfp:1}, we denote
\begin{align*}
J( x_0 ) :=\left\{ j\in J\,:\,g_j(x_0) =0 \right\},\,\,
J( x_0,v ) :=\left\{ j\in J(x_0) \,:\, g^\circ_j( x_0; v)=0 \right\},
\end{align*}
and
\begin{align}\label{def_of_Q}
Q:=\left\{ x\in Z \,:\, g( x ) \leqq 0,\, h( x ) = 0,\, f( x ) \leqq_C s\ast F(x) \right\},
\end{align}
where  $s= \left( \frac{f}{F} \right) (x_0) $ and  $g^\circ_j( x_0; v)$ is the
 Clarke directional derivative (Definition \ref{def:3}) of $g_{j}$ at $x_0$ in the direction $v\in Z$.

\begin{definition}\label{def:7} \cite{Jahn,Clarke1983}
Let $X$ and $Q$ be as in \eqref{def_of_X} and \eqref{def_of_Q}, respectively,  and $X\ne\emptyset$.
\begin{enumerate}[(i)]
  \item
  The \emph{contingent cone} of $X$ at $x_0\in  X$ is defined by
  \begin{equation*}
  T( X,x_0) :=\left\{ d\in Z \,:\, \exists\, t_k\rightarrow 0^+,\, \exists\, d_k\rightarrow d  \,\text{ such that }\,x_0+t_kd_k\in X,\,\forall\, k\in \mathbb{N} \right\}.
  \end{equation*}

  \item
  The \emph{linearizing cone} of $Q$ at $x_0\in X$ is defined by
\begin{equation*}
C\left( Q,x_0 \right)
 := \left\{\begin{array}{lll}
    d\in Z \, : & \left( \frac{f}{F} \right)^\circ_i ( x_0;d ) \leq 0, &~~ i\in I, \\
    &~  g_j^\circ ( x_0; d) \leq 0, &~~ j\in J(x_0), \\
                          &~  h_k^\circ ( x_0;d ) =0, &~~ k \in K
    \end{array}\right\}.
\end{equation*}
\end{enumerate}
\end{definition}
The set $T \left(X, x_0 \right)$ is a nonempty and closed cone with $0\in T( X,x_0)$.
It is well-known that  if $x_0\in \mbox{cl} X_1\subseteq \mbox{cl} X_2 \subseteq X$,
 then $T( X_1,x_0)\subseteq T( X_2,x_0 )$; see e.g., \cite{Jahn}.

\begin{definition}
A vector $v\in Z$ is called a \emph{critical point} at $x_0\in X$ iff
$$v\in T(X,x_0)\bigcap C(Q,x_0).$$
\end{definition}

\noindent
We denote the \emph{set of all critical points} at $x_0 \in X$ by $D(x_0)$, i.e.,
\begin{align*}
 D(x_0):=\{ v\in Z \,:\,   v\in T(X,x_0)\bigcap C(Q,x_0)  \}.
\end{align*}

\begin{definition}\cite{Ivanov2015,Pales1994}\label{def:2}
Let  $\vartheta: U\rightarrow \mathbb{R}$ be a function and $x_0\in U$.
\begin{enumerate}[(i)]
\item
The \emph{second-order directional derivative} of $\vartheta$ at $x_0$ in the direction $v_0 \in Z$ is defined by
$$\vartheta''\left( x_0; v_0 \right) := \lim_{t\rightarrow 0^+}\frac{\vartheta\left( x_0 + tv_0 \right) -\vartheta\left( x_0 \right) -t\vartheta' \left( x_0; v_0 \right)}{\tfrac{1}{2}t^2}, \text{ provided limit exists}. $$
\item
 The \emph{P\'ales-Zeidan second-order generalized directional derivative} of $\vartheta$ at $x_0$ in the direction $v_0 \in Z$ is defined by
$$\vartheta^{\circ\circ}\left( x_0; v_0 \right) := \limsup_{t\rightarrow 0^+}\frac{\vartheta\left( x_0 + tv_0 \right) -\vartheta\left( x_0 \right) -t\vartheta^{\circ}\left( x_0; v_0 \right)}{\tfrac{1}{2}t^2}.$$
\end{enumerate}
\end{definition}

\begin{remark}\label{rem:2.2} \cite{Pales1994}
\begin{enumerate}[(i)]
\item
If $\vartheta$ is twice Fr\'echet differentiable at $x_0$, then
\begin{align*}
\langle \nabla ^2 \vartheta( x_0 ) v, v \rangle = \vartheta''( x_0;v )=\vartheta^{\circ\circ}( x_0; v), \,\forall\, v \in Z.
\end{align*}

\item
If $\vartheta'( x_0; v_0)=\vartheta^{\circ}( x_0; v_0 )$ and $\vartheta''( x_0; v_0 )$ exists, then $\vartheta''( x_0; v_0 )=\vartheta^{\circ\circ}( x_0; v_0)$. Besides, for any $\beta > 0$, $(\beta \vartheta)'( x_0; v_0)= \beta \vartheta'( x_0; v_0)$, $(\beta\vartheta)^{\circ}( x_0; v_0)= \beta\vartheta^{\circ}( x_0; v_0)$,  $(\beta \vartheta)''( x_0; v_0)= \beta \vartheta''(x_0; v_0)$ and $(\beta\vartheta)^{\circ\circ}(x_0; v_0)= \beta\vartheta^{\circ\circ}(x_0; v_0)$.
\end{enumerate}
\end{remark}

\begin{definition}\label{def:4}\cite{Clarke1983}
A function $\vartheta: U \rightarrow \mathbb{R}$ is called locally Lipschitz continuous at $x_0 \in U$ if there exists a neighborhood $V(x_0)$ and a positive constant $L(x_0)$ such that
\begin{align*}
| \vartheta( y ) -\vartheta( z ) | \leq L(x_0) \| y-z \|,\,\forall\, y,z\in  V(x_0) \cap U.
\end{align*}
The function $\vartheta$ is called locally Lipschitz continuous on $U$ if it is locally Lipschitz continuous at every $x\in U$. In particular, if the positive constant $L(x)$ is independent of $x\in U$, then $\vartheta$ is called Lipschitz continuous on $U$.
\end{definition}

\begin{definition}\label{def:5}\cite{Clarke1983}
Let $\vartheta: U \rightarrow \mathbb{R}$ be locally Lipschitz continuous on $U$. The Clarke subdifferential of $\vartheta$ at $x_0\in U$ is defined by
\begin{equation}\label{cwf:1}
\partial \vartheta( x_0 ) := \left\{\xi\in Z^{\ast} \,:\, \langle \xi, v \rangle \le \vartheta^{\circ}( x_0; v),\,\forall\, v\in Z \right\}.
\end{equation}
\end{definition}

\begin{lemma}\label{yl:7} {\rm\cite{Ivanov2010}}
Let $\vartheta: U \rightarrow \mathbb{R}$ be locally Lipschitz continuous at $x_0 \in U$. Then there exists a constant $L_0>0$ such that $\| \xi \| \le L_0$ for arbitrary $\xi \in \partial \vartheta(x_0)$.
\end{lemma}

\begin{remark}\label{regular:Lip}
\begin{enumerate}
\item[{\rm (i)}]
    The existence of $\vartheta^{\circ}( x_0;v_0 )$ does not necessarily imply the existence of $\vartheta'( x_0;v_0 )$, and even if $\vartheta'( x_0;v_0 )$ exists, they may not be equal. However, if $\vartheta$ is a continuously Fr\'echet differentiable function at $x_0$ and the Fr\'echet derivative is $\nabla \vartheta( x_0 )$, then it is regular in the sense of Clarke, i.e.,
    $$\vartheta'( x_0;v_0 )=\vartheta^{\circ}( x_0;v_0 )= \langle \nabla \vartheta( x_0 ), v_0\rangle = \langle \vartheta'_G( x_0 ), v_0\rangle.$$

\item[{\rm (ii)}]
    If $\vartheta$ is locally Lipschitz continuous on $U$, then $\vartheta^{\circ}( x;v )$ exists (finite) for any $x\in U$ and $v\in Z$, which implies that $\partial \vartheta( x )$ is compact. Besides, if $\vartheta_{1},\vartheta_{2}: U \rightarrow \mathbb{R}$ are locally Lipschitz continuous and regular in the sense of Clarke at $x_0\in U$ and $\vartheta_{2}(x)\neq 0$ for all $x\in U$, then $\frac{\vartheta_{1}}{\vartheta_{2}}: U \rightarrow \mathbb{R}$ is regular in the sense of Clarke at $x_0\in U$, i.e.,  for each $v\in Z$,
    \begin{align*}
     \left( \frac{\vartheta_{1}}{\vartheta_{2}} \right)^\circ( x_0; v)
     =\left( \frac{\vartheta_{1}}{\vartheta_{2}} \right)'( x_0; v)
     =\frac{\vartheta'_{1}(x_0;v)}{\vartheta_{2}(x_0)}-  \frac{\vartheta_{1}(x_0)}{\vartheta^{2}_{2}(x_0)} \vartheta'_{2}(x_0;v).
    \end{align*}
\end{enumerate}
\end{remark}

Next, before we end the section, we present chain rules of quotient functions involving locally Lipschitz functions in terms of generalized directional derivatives.

\begin{proposition}\label{pro:2.1}
Let $\vartheta_{1}$ and $\vartheta_{2}$ be two real-valued locally Lipschitz continuous functions on $U$ with $\vartheta_{2}$ being positive-valued.
For any $x_0\in U$, $v\in Z$ and $\beta \in \mathbb{R}_{+}$, the following are true:
\begin{enumerate}
\item[{\rm (i)}] $\vartheta_{1}- \beta \vartheta_{2}$ and $\frac{\vartheta_{1}}{\vartheta_{2}}$ are locally Lipschitz continuous on $U$;

\noindent
If, in addition, $\vartheta_{2}$ is G\^ateaux differentiable at $x_0$ with G\^ateaux derivative $\vartheta'_{2,G}(x_0)$, then

\item[{\rm (ii)}] $\left( \frac{\vartheta_{1}}{\vartheta_{2}} \right)^\circ \left( x_0; v \right)=\frac{1}{\vartheta_{2}(x_0)}\left[\vartheta_{1}^\circ(x_0;v)-\frac{\vartheta_{1}(x_0)}{\vartheta_{2}(x_0)} \langle \vartheta'_{2,G}(x_0), v \rangle\right]$;

\item[{\rm (iii)}]  $(\vartheta_{1}- \beta \vartheta_{2})^\circ(x_0;v)=\vartheta_{1}^\circ(x_0;v)-\beta \langle  \vartheta'_{2,G}(x_0), v\rangle$;

\noindent
Moreover, if $\vartheta''_{2}(x_0;v)$ exists, then
\item[{\rm (iv)}] $\left( \frac{\vartheta_{1}}{\vartheta_{2}} \right)^{\circ\circ} \left( x_0;v \right)=\frac{1}{\vartheta_{2}(x_0)}\left[\vartheta_{1}^{\circ\circ}(x_0;v)-\frac{\vartheta_{1}(x_0)}{\vartheta_{2}(x_0)}\vartheta''_{2}\left( x_0;v \right)\right]$;

\item[{\rm (v)}] $(\vartheta_{1}-\beta \vartheta_{2})^{\circ\circ}(x_0;v)=\vartheta_{1}^{\circ\circ}(x_0;v)-\beta \vartheta''_{2}\left( x_0;v \right)$.
\end{enumerate}
\end{proposition}

\begin{proof}
(i) By calculation, one has
\begin{align*}
\lVert \left( \vartheta_{1}-\beta \vartheta_{2} \right) \left( x \right) -\left( \vartheta_{1}-\beta \vartheta_{2} \right) \left( z \right) \rVert
&=\lVert \vartheta_{1}\left( x \right) -\vartheta_{1}\left( z \right) -\beta \vartheta_{2} \left( x \right) +\beta \vartheta_{2}\left( z \right) \rVert\\
&\le \lVert \vartheta_{1}\left( x \right) -\vartheta_{1}\left( z \right) \rVert +\beta\lVert \vartheta_{2}\left( x \right) -\vartheta_{2}\left( z \right) \rVert,
\end{align*}
and
\begin{align*}
\left\lVert \left( \frac{\vartheta_{1}}{\vartheta_{2}} \right) \left( x \right) -\left( \frac{\vartheta_{1}}{\vartheta_{2}} \right) \left( z \right) \right\rVert
&=\left\lVert \frac{\vartheta_{1}\left( x \right) -\vartheta_{1}\left( z \right)}{\vartheta_{2}\left( x \right)}-\frac{\vartheta_{1}\left( z \right) \left[ \vartheta_{2}\left( x \right) -\vartheta_{2}\left( z \right) \right]}{\vartheta_{2}\left( x \right) \vartheta_{2}\left( z \right)} \right\rVert\\
&\le \left\lVert \frac{\vartheta_{1}\left( x \right) -\vartheta_{1}\left( z \right)}{\vartheta_{2}\left( x \right)} \right\rVert +\left\| \frac{\vartheta_{1}\left( z \right) \left[ \vartheta_{2}\left( x \right) -\vartheta_{2}\left( z \right) \right]}{\vartheta_{2}\left( x \right) \vartheta_{2}\left( z \right)} \right\|.
\end{align*}
By virtue of Definition \ref{def:4} and $\vartheta_{2}(x), \vartheta_{2}(z)>0$, we obtain that $\vartheta_{1}-\beta \vartheta_{2}$ and  $\frac{\vartheta_{1}}{\vartheta_{2}}$ are locally Lipschitz continuous on $U$. \\

\noindent
(ii)--(iv) The proofs are analogous to \cite[Theorem 3.1 and Theorem 3.2]{VanSu2023} .

\noindent
(v) If $\vartheta''_{2}(x_0;v)$ exists, Definition \ref{def:2} and Remark \ref{rem:2.2} imply that
\begin{align*}
&~~~~( \vartheta_{1}- \beta \vartheta_{2} )^{\circ \circ}( x_0;v ) \\
&=\limsup_{t\rightarrow 0^+}\frac{( \vartheta_{1}- \beta \vartheta_{2} ) ( x_0+tv ) -( \vartheta_{1}- \beta \vartheta_{2} )( x_0 ) - t( \vartheta_{1}-\beta \vartheta_{2} ) ^{\circ}( x_0;v )}{\tfrac{1}{2}t^2}\\
&=\limsup_{t\rightarrow 0^+}\left[ \frac{( \vartheta_{1} - \beta \vartheta_{2} ) ( x_0+tv ) -( \vartheta_{1}- \beta \vartheta_{2})( x_0 )}{\tfrac{1}{2}t^2}
-\frac{t( \vartheta_{1}- \beta \vartheta_{2})^{\circ}( x_0;v)}{\tfrac{1}{2}t^2} \right] \\
&=\limsup_{t\rightarrow 0^+}\left[ \frac{\left( \vartheta_{1}-\beta \vartheta_{2} \right) \left( x_0+tv \right) -\left( \vartheta_{1}-\beta \vartheta_{2} \right) \left( x_0 \right)}{\tfrac{1}{2}t^2}-\frac{t\vartheta_{1}^{\circ}\left( x_0;v \right) -t \beta \langle \vartheta'_{2,G}\left( x_0 \right), v\rangle}{\tfrac{1}{2}t^2} \right] \\
&=\limsup_{t\rightarrow 0^+}\frac{\vartheta_{1}\left( x_0+tv \right) -\vartheta_{1}\left( x_0 \right) - t \vartheta_{1}^{\circ}\left( x_0;v \right)}{\tfrac{1}{2}t^2}- \beta \lim_{t\rightarrow 0^+} \frac{\vartheta_{2}\left( x_0+tv \right) - \vartheta_{2}\left( x_0 \right) -\langle \vartheta'_{2,G}\left( x_0 \right), v\rangle }{\tfrac{1}{2}t^2}\\
&=\vartheta_{1}^{\circ \circ}\left( x_0;v \right) -\beta \vartheta''_{2}\left( x_0;v \right).
\end{align*}
The proof is completed.
\end{proof}

\begin{remark}
In the chain rules for quotient functions involving locally Lipschitz functions presented in \cite{VanSu2023}, the function $\vartheta_{2}$ is taken to be strictly differentiable on $U$ in the sense of Clarke. In comparison, Proposition \ref{pro:2.1} just assumed the existence of the G\^ateaux derivative $\vartheta'_{2,G}(x_0)$. Although in finite-dimensional normed spaces, the Clarke strict derivative $D_s \vartheta_{2}(x)$ and the G\^ateaux derivative $\vartheta'_{2,G}(x)$ coincide, Proposition \ref{pro:2.1} being in a general Banach space is a stronger result than Theorem 3.2 in \cite{VanSu2023}. Furthermore, it is worth noting that Proposition \ref{pro:2.1} (iv) and (v) may not be true if $\vartheta''_{2}(x;v)$ is replaced by $\vartheta_{2}^{\circ\circ}(x;v)$ (see Example \ref{second:dir}).
\end{remark}

\begin{example}\label{second:dir}
Let $\vartheta_{1}, \vartheta_{2}: \mathbb{R}\rightarrow  \mathbb{R}$ be  two functions as follows:
$\vartheta_{1}(x)=x^2+1$
and
\begin{eqnarray*}
\vartheta_{2}\left( x \right) =\left\{ \begin{array}{ll}
	x^2\sin \displaystyle\frac{1}{x}+1, &\hbox{if $x\ne0$},\\
	1, &\hbox{if $x=0$}.\\
\end{array} \right.
\end{eqnarray*}
It is easy to verify that both $\vartheta_{1}$ and $\vartheta_{2}$ are locally Lipschitz continuous on the open interval $U=(-1,1)$. Also, $\vartheta_{2}$ is G\^ateaux differentiable on $U$ and its G\^ateaux derivative is
\[\vartheta'_{2,G}\left( x \right) =\left\{ \begin{array}{ll}
	2x\sin \displaystyle\frac{1}{x}-\cos \frac{1}{x},&\hbox{if $x\ne0$};\\
	0,&\hbox{if $x=0$}.\\
\end{array} \right. \]
After calculation, we obtain that
\begin{align*}
  &\vartheta_{1}^\circ(0;v)=0,  \,\, \vartheta_{1}^{\circ\circ}(0;v)=2v^2,\,\, \vartheta_{2}^\circ(0;v)=\vartheta'_{2}(0;v)=\vartheta'_{2,G}(0)v=0, \\
  & (\vartheta_{1}-\vartheta_{2})^\circ(0;v)=0,\,\,  \vartheta_{2}^{\circ\circ}(0;v)=2v^2, \text{ and } \vartheta''_{2}(0;v) \text{ does not exist}.
\end{align*}
So,
\begin{equation}\label{ex2.1}
\frac{1}{\vartheta_{2}(0)}\left[\vartheta_{1}^{\circ\circ}(0;v)-\frac{\vartheta_{1}(0)}{\vartheta_{2}(0)}\vartheta_{2}^{\circ\circ}\left( 0;v \right)\right]=2v^2-2v^2=0
\end{equation}
and $\vartheta_{1}^{\circ\circ}(0;v)-\vartheta_{2}^{\circ\circ}(0;v)=0$.
However, by Definition \ref{def:2} (ii) and Definition \ref{def:3} (i), one has
\begin{align}\label{ex2.2}
  \left(\frac{\vartheta_{1}}{\vartheta_{2}}\right)^{\circ}(0;v)=0,\,\,
  \left(\frac{\vartheta_{1}}{\vartheta_{2}}\right)^{\circ\circ}(0;v)=4v^2 ~\text{ and }~ (\vartheta_{1}-\vartheta_{2})^{\circ\circ}(0;v)=4v^2.
\end{align}
So, for any $v\neq 0$,  \eqref{ex2.1} and \eqref{ex2.2}  imply that
 \begin{align*}
  \left(\frac{\vartheta_{1}}{\vartheta_{2}}\right)^{\circ\circ}(0;v)> \frac{1}{\vartheta_{2}(0)}\left[\vartheta_{1}^{\circ\circ}(0;v)-\frac{\vartheta_1(0)}{\vartheta_{2}(0)}\vartheta_{2}^{\circ\circ}\left( 0;v \right)\right]
\end{align*}
and
 \begin{align*}
  (\vartheta_{1}-\vartheta_{2})^{\circ\circ}(0;v)> \vartheta_{1}^{\circ\circ}(0;v)-\vartheta_{2}^{\circ\circ}(0;v)=0.
\end{align*}
Therefore, $\vartheta''_{2}(x;v)$  cannot be replaced by $\vartheta_{2}^{\circ\circ}(x;v)$ in Proposition \ref{pro:2.1} (iv) and (v).
\end{example}

\begin{lemma}\label{yl:1}
Let $\vartheta_{1},\vartheta_{2}: U\rightarrow \mathbb{R}$ be locally Lipschitz continuous and regular in the sense of Clarke at $x_{0}\in U$, $\vartheta_{2}$ be positive-valued, G\^ateaux differentiable on $U$ and $\left(\frac{\vartheta_{1}}{\vartheta_{2}}\right)^{\circ\circ}\left( x_{0};v \right)$ exist in the direction $v\in Z$. Assume that $x_n\rightarrow x_0$, $t_n\rightarrow 0^+$, $r_n\rightarrow r$ and $r_n>0$ for all $n \in \mathbb{N}$. If $w_n = \frac{x_n-x_0-t_nv}{\tfrac{1}{2}r_{n}^{-1}t_{n}^{2}}\rightarrow w$, then
\begin{equation*}
  \begin{split}
     & \limsup_{n\rightarrow \infty}\frac{\left(\frac{\vartheta_{1}}{\vartheta_{2}}\right)\left( x_n \right) -\left(\frac{\vartheta_{1}}{\vartheta_{2}}\right)\left( x_0 \right) -t_n\left(\frac{\vartheta_{1}}{\vartheta_{2}}\right)^{\circ}\left( x_0;v \right)}{\tfrac{1}{2}r_{n}^{-1}t_{n}^{2}}\\
     \leq & \left(\frac{\vartheta_{1}}{\vartheta_{2}}\right)^{\circ}\left( x_0;w \right) +r\left(\frac{\vartheta_{1}}{\vartheta_{2}}\right)^{\circ\circ}\left( x_0;v \right)\\
     = & \left(\frac{\vartheta_{1}}{\vartheta_{2}}\right)'\left( x_0;w \right) + r\left(\frac{\vartheta_{1}}{\vartheta_{2}}\right)^{\circ\circ}\left( x_0;v \right).
   \end{split}
\end{equation*}
Further, if $\left(\frac{\vartheta_{1}}{\vartheta_{2}}\right)''\left( x_{0};v \right)$ exists in the direction $v\in Z$, then
\begin{equation}\label{le:2.3:2}
\limsup_{n\rightarrow \infty}\frac{\left(\frac{\vartheta_{1}}{\vartheta_{2}}\right)\left( x_n \right) -\left(\frac{\vartheta_{1}}{\vartheta_{2}}\right)\left( x_0 \right) -t_n\left(\frac{\vartheta_{1}}{\vartheta_{2}}\right)^{\circ}\left( x_0;v \right)}{\tfrac{1}{2}r_{n}^{-1}t_{n}^{2}} \leq \left(\frac{\vartheta_{1}}{\vartheta_{2}}\right)'\left( x_0;w \right) +r\left(\frac{\vartheta_{1}}{\vartheta_{2}}\right)''\left( x_0;v \right).
\end{equation}
\end{lemma}

\begin{proof}
Let $x_n=x_0+t_nv + \displaystyle \tfrac{1}{2} r_{n}^{-1} t_{n}^{2}w_n$.
From Remark \ref{regular:Lip} (ii) and Proposition \ref{pro:2.1} (i), it yields that $\frac{\vartheta_{1}}{\vartheta_{2}}$ is locally Lipschitz continuous and regular in the sense of Clarke at $x_{0}\in U$.
By Lebourg's Mean-Value Theorem (see \cite[Theorem 2.3.7]{Lebourg1975}),
 we get that there exist $\theta_n\in ( 0,1 )$,
  $\xi _n\in \partial \left(\frac{\vartheta_{1}}{\vartheta_{2}}\right)\left( x_0+t_nv+\frac{\theta _n}{2}r_{n}^{-1}t_n^2w_n \right)$ such that
\begin{equation*}
\left(\frac{\vartheta_{1}}{\vartheta_{2}}\right)\left( x_n \right) -\left(\frac{\vartheta_{1}}{\vartheta_{2}}\right)\left( x_0+t_nv \right) =\tfrac{1}{2}r_{n}^{-1}t_{n}^{2} \langle \xi_{n}, w_n \rangle.
\end{equation*}
Combining Lemma \ref{yl:7} with the Lipschitz continuity of $\frac{\vartheta_{1}}{\vartheta_{2}}$,
there exists a constant $L_0>0$ such that $\| \xi_n \| \le L_0$ for $n$ sufficiently large. Without loss of generality, assume that $\xi _n\rightarrow \xi$. Owing to the closedness of $\partial \left(\frac{\vartheta_{1}}{\vartheta_{2}}\right)\left( \cdot \right)$ (see \cite[Theorem 8.6]{Rockafellar1998}), there exists $\xi \in \partial \left(\frac{\vartheta_{1}}{\vartheta_{2}}\right)\left( x_0 \right)$ such that
\begin{align*}
\lim_{n\rightarrow \infty}\frac{\left(\frac{\vartheta_{1}}{\vartheta_{2}}\right)\left( x_n \right) -\left(\frac{\vartheta_{1}}{\vartheta_{2}}\right)\left( x_0+t_nv \right)}{\tfrac{1}{2}r_{n}^{-1}t_{n}^{2}}= \langle \xi, w\rangle.
\end{align*}
In turn, due to $0<r_n\rightarrow r$, using Definition \ref{def:2} (ii)   we get
\begin{align*}
\limsup_{n\rightarrow \infty}\frac{\left(\frac{\vartheta_{1}}{\vartheta_{2}}\right)\left( x_0+t_nv \right) -\left(\frac{\vartheta_{1}}{\vartheta_{2}}\right)\left( x_0 \right) -t_n\left(\frac{\vartheta_{1}}{\vartheta_{2}}\right)^{\circ}\left( x_0 ;v \right)}{\tfrac{1}{2}r_{n}^{-1}t_{n}^{2}}\le r\left(\frac{\vartheta_{1}}{\vartheta_{2}}\right)^{\circ\circ}\left( x_0;v \right),
\end{align*}
which together with \eqref{cwf:1} yields that
\begin{align*}
~&~\limsup_{n\rightarrow \infty}\frac{\left(\frac{\vartheta_{1}}{\vartheta_{2}}\right)\left( x_n \right) -\left(\frac{\vartheta_{1}}{\vartheta_{2}}\right)\left( x_0 \right) -t_n\left(\frac{\vartheta_{1}}{\vartheta_{2}}\right)^{\circ}\left( x_0 ;v \right)}{\tfrac{1}{2}r_{n}^{-1}t_{n}^{2}}\\
\leq ~&~ \limsup_{n\rightarrow \infty}\frac{\left(\frac{\vartheta_{1}}{\vartheta_{2}}\right)\left( x_n \right) -\left(\frac{\vartheta_{1}}{\vartheta_{2}}\right)\left( x_0+t_nv \right)}{\tfrac{1}{2}r_{n}^{-1}t_{n}^{2}}\\
~&~ + \limsup_{n\rightarrow \infty}\frac{\left(\frac{\vartheta_{1}}{\vartheta_{2}}\right)\left( x_0+t_nv \right) -\left(\frac{\vartheta_{1}}{\vartheta_{2}}\right)\left( x_0 \right) -t_n\left(\frac{\vartheta_{1}}{\vartheta_{2}}\right)^{\circ}\left( x_0 ;v\right) }{\tfrac{1}{2}r_{n}^{-1}t_{n}^{2}}\\
\leq ~&~ \langle \xi, w\rangle + r\left(\frac{\vartheta_{1}}{\vartheta_{2}}\right)^{\circ\circ}(x_0;v)\\
\leq ~&~ \left(\frac{\vartheta_{1}}{\vartheta_{2}}\right)^{\circ}\left( x_0;w \right) + r\left(\frac{\vartheta_{1}}{\vartheta_{2}}\right)^{\circ\circ}(x_0;v)\\
= ~&~ \left(\frac{\vartheta_{1}}{\vartheta_{2}}\right)'\left( x_0;w \right) + r\left(\frac{\vartheta_{1}}{\vartheta_{2}}\right)^{\circ\circ}(x_0;v),
\end{align*}
which comes from the Clarke regularity of $\frac{\vartheta_{1}}{\vartheta_{2}}$ at $x_{0}\in U$. If $\left(\frac{\vartheta_{1}}{\vartheta_{2}}\right)''\left( x_{0};v \right)$ exists in the direction $v\in Z$, then
\eqref{le:2.3:2} results from
Remark \ref{rem:2.2} (ii).
The proof is completed.
\end{proof}

\begin{remark}
  It should be pointed out that Lemma \ref{yl:1} improves Lemma 2.1 of \cite{Feng2019} even when $\vartheta_{1}$ is continuously Fr\'echet differentiable at $x_{0}$, $\vartheta_{2}$ is a positive constant, and  $f^{\circ\circ}_{1}(x_{0};v)$ exists in the direction $v\in Z$. Moreover,  Lemma \ref{yl:1} is reduced to Lemma 2.1 of \cite{Feng2019} when $\vartheta_{1}$ is continuously Fr\'echet differentiable at $x_{0}$, $\vartheta_{2}$ is a positive constant, and $\vartheta''_{1}(x_{0};v)$ exists along $v\in Z$.
\end{remark}

\section{Second-order strong KKT necessary conditions for NMFP}\label{sec:3}

If the multipliers of all objective functions in the KKT conditions are positive, it is that the strong KKT conditions hold. The strong KKT condition basically implies that all objective functions are active at the point at which the necessary optimality conditions hold. This section introduces a generalized second-order Abadie regularity condition, which extends the second-order Abadie regularity conditions proposed in \cite{Feng2019,Aanchal2023}. We further study second-order strong KKT necessary conditions that contain some new information on multipliers and imply the strong KKT necessary conditions for Borwein-type properly efficient solutions of NMFP, which is stronger than (weak) Pareto efficient solution.

To begin with, we recollect the notions of Borwein properly efficient solution, projective second-order tangent cone, and projective second-order linearizing cone.

\begin{definition}\label{bpe:1}\cite{Ehrgott2005,Borwein1977}
A point $x_0\in X$ is said to be \emph{Borwein properly efficient solution} of NMFP iff
$$T\left( \left(\frac{f}{F}\right)\left( X \right) + \mathbb{R}^{p}_{+},\left(\frac{f}{F}\right)\left( x_0 \right) \right) \bigcap \left( - \mathbb{R}^{p}_{+} \right) =\left\{ 0 \right\}.$$
\end{definition}

It is pointed out in \cite{Sawaragi1985} that a Borwein properly efficient solution defined in Definition \ref{bpe:1} must also be a Pareto efficient solution.

\begin{definition}\label{def:7} \cite{Feng2019}
Let $X, Q\subseteq Z$ and $v \in Z$.
\begin{enumerate}[(i)]
  \item
   The \emph{projective second-order tangent cone} of $X$ at $x_0 \in X$ in the direction $v$ is defined by
\[
\widetilde{T}^2\left( X,x_0,v \right) =
\begin{aligned}
     &\left\{\begin{array}{ll}
             \left( w,r \right) \in Z \times \mathbb{R}_+:\,  & \exists \,t_k\rightarrow 0^+,\,\exists\, r_k\rightarrow r,\,\exists\, w_k\rightarrow w \,\text{ such that } \\                &\frac{t_k}{r_k}\rightarrow 0^+,\,\, x_0+t_kv+\tfrac{1}{2}t_{k}^{2}w_k\in X,\,\forall\, k\in \mathbb{N}
            \end{array}\right\}.
\end{aligned}
\]

\item
The \emph{projective second-order linearizing cone} of $Q$ at $x_0 \in X$ in the direction $v$ is defined by
\begin{equation*}
\resizebox{0.9\textwidth}{!}{
$
\widetilde{C}^2\left( Q,x_0,v \right) =
\begin{aligned}
\left\{\begin{array}{lll}
(w,r)\in Z \times \mathbb{R}_+:
& \left( \frac{f}{F} \right)^\circ _i \left( x_0;w \right)+r\left( \frac{f}{F} \right)^{\circ \circ}_i\left( x_0;v \right) \leq 0,  & \, i\in I, \\
& g_j^\circ ( x_0;w )+r g_{j}^{\circ \circ}( x_0;v ) \leq 0, & \, j\in J( x_0,v ), \\
& h_k^\circ ( x_0;w )+rh_k^{\circ \circ}( x_0;v ) = 0, &\,  k\in K
\end{array}\right\}.
\end{aligned}
$}
\end{equation*}

\end{enumerate}
\end{definition}

The  projective second-order tangent cone $\widetilde{T}^2\left( X,x_0,v \right)$ has been widely applied  to study optimality conditions; see \cite{Feng2019,Aanchal2023,Penot1999,Jimenez2004,Feng2018}.
 Moreover, it follows from \cite[Proposition 2.1]{Feng2018} that $\widetilde{T}^2\left(X,x_0,0 \right)=T(X,x_0)\times \mathbb{R}_+$, and $v\notin T(X,x_0)$ implies $\widetilde{T}^2\left( X,x_0,0 \right)= \emptyset$.

Next, we introduce a new second-order Abadie regularity condition and a new second-order Guignard-type regular condition for NMFP.
\begin{definition}
We say that
 \begin{enumerate}[(i)]
  \item
    a \emph{generalized second-order Abadie-type regularity condition} (GSOARC) holds at $x_0\in X$ in the direction $v\in D(x_0)$ iff
\begin{equation}\label{abd:1}
\widetilde{C}^2 ( Q, x_0, v)\subseteq \widetilde{T}^2 ( X,x_0,v ).
\end{equation}

\item
  a \emph{generalized second-order Guignard-type regular condition} (GSOGRC) holds at $x_0\in X$ in the direction $v\in D(x_0)$ iff
\begin{equation}\label{gui:1}
\widetilde{C}^2 ( Q, x_0, v)\subseteq \text{clco}\,\widetilde{T}^2( X, x_0, v).
\end{equation}
\end{enumerate}
\end{definition}

\begin{remark}\label{rem:1}
\begin{enumerate}[(i)]
\item  It is easy to see that GSOGRC is a weaker regular condition than GSOARC.
    If $v=0$, specifically, the P\'ales-Zeidan second-order generalized directional derivative becomes $0$, then GSOARC is reduced to a degenerate form
\begin{equation}\label{abd:3}
C\left( Q,x_0 \right) \subseteq T\left( X,x_0 \right),
\end{equation}
as $\widetilde{C}^2\left( Q,x_0,0 \right)=C\left( Q,x_0 \right)\times \mathbb{R}_+$ and $\widetilde{T}^2\left( X,x_0,0 \right)=T\left( X,x_0 \right)\times \mathbb{R}_+$.

  \item
    If all the functions involved have continuous Fr\'echet derivatives and the corresponding second-order directional derivatives in the direction $v$ exist, then the GSOARC is reduced to the \emph{second-order Abadie regularity condition} SOARC in \cite{Feng2019}.

   \item  In \cite{Feng2019}, SOARC was proposed by utilizing Fr\'echet derivatives and second-order directional derivatives. Another second-order Abadie constraint qualification in \cite{Aanchal2023} was introduced in terms of Clarke generalized directional derivatives and P\'ales-Zeidan second-order generalized directional derivatives, which did not involve the objective functions. Compared with the second-order Abadie regularity conditions in \cite{Feng2019,Aanchal2023}, GSOARC is introduced in terms of Clarke generalized directional derivatives and P\'ales-Zeidan second-order generalized directional derivatives, which also incorporate objective functions.
\end{enumerate}
\end{remark}

\begin{framed}
\begin{assumption}\label{js:1}
\begin{enumerate}
\item[{\rm (i)}] For $i\in I,\, j\in J,\, k\in K$, the functions $f_i$, $F_i$, $g_j$ and $h_k$ are locally Lipschitz continuous and regular in the sense of Clarke on $U$, and $f_i^{\circ\circ}(x;v)$, $g_j^{\circ\circ}(x;v)$ and $h_k^{\circ\circ}(x;v)$ are finite on $U$ for all direction $v\in Z$.

\item[{\rm (ii)}] For $i\in I$, $F_i$'s are G\^ateaux differentiable on $U$ with G\^ateaux derivative $F'_{i,G}$, and $F''_i(x;v)$ exist on $U$ for each direction $v\in Z$.
\end{enumerate}
\end{assumption}
\end{framed}

\begin{lemma}\label{yl:3}
Let $x_0 \in X$ be a Borwein properly efficient solution of {\rm NMFP}, $v\in D(x_0)$ and $s=(s_1,s_2,\ldots,s_p)$ be defined as that in Lemma \ref{lem:2.1}. Let Assumption ~\ref{js:1} be fulfilled. Then, the following system
\begin{equation}\label{xt:1}
\left\{ \begin{array}{l}
	f_{i}^{\circ}\left( x_0;w \right) +rf_{i}^{\circ\circ}\left( x_0;v \right)-s_i(\langle F'_{i,G}\left( x_0 \right), w \rangle + rF''_{i}\left( x_0;v \right)) \leq 0,\,\forall\, i\in I,\\
	f_{i}^{\circ}\left( x_0;w \right) + r f_{i}^{\circ\circ}\left( x_0; v \right) - s_i(\langle F'_{i,G}\left( x_0 \right), w \rangle + r F''_{i}\left( x_0;v \right)) <0,\, \exists\, i\in I,\\
	\left( w,r \right) \in \widetilde{T}^2\left( X,x_0,v \right),
\end{array} \right.
\end{equation}
has no solution $(w,r) \in Z \times \mathbb{R}_+$.
\end{lemma}

\begin{proof}
From Proposition \ref{pro:2.1} (ii) and (iv), it follows that
\begin{align*}
~&~\left( \frac{f_i}{F_i} \right)^\circ \left( x_0;w \right)+r\left( \frac{f_i}{F_i} \right)^{\circ\circ} \left( x_0;v \right)\\
=~&~\frac{1}{F_i(x_0)}\left[f_i^\circ(x_0;w)-s_i \langle F'_{i,G}(x_0), w\rangle \right] + \frac{r}{F_i(x_0)}\left[f_i^{\circ\circ}(x_0;v)-s_iF''_i\left( x_0;v \right)\right]\\
=~&~\frac{1}{F_i(x_0)}\left[f_i^\circ(x_0;w)+rf_i^{\circ\circ}(x_0;v)-s_i(\langle F'_{i,G}\left( x_0 \right), w \rangle +rF''_{i}\left( x_0;v \right))\right].
\end{align*}
Due to $F_i(x_0) > 0$ for all $i\in I$, it ensures that
\begin{align*}
 ~&~ \left( \frac{f_i}{F_i} \right)^\circ \left( x_0;w \right)+r\left( \frac{f_i}{F_i} \right)^{\circ\circ} \left( x_0;v \right)<(\le) 0 \\
\Longleftrightarrow ~&~
f_i^\circ(x_0;w)+rf_i^{\circ\circ}(x_0;v)-s_i(F'_{i,G}\left( x_0 \right) w+rF''_{i}\left( x_0;v \right))<(\le)0.
\end{align*}
Therefore, the system \eqref{xt:1} is equivalent to
\begin{equation}\label{xt:1.1}
\left\{ \begin{array}{l}
	\left( \frac{f_i}{F_i} \right)^\circ \left( x_0;w \right)+r\left( \frac{f_i}{F_i} \right)^{\circ\circ} \left( x_0;v \right) \leq 0,\,\forall\, i\in I,\\
	\left( \frac{f_i}{F_i} \right)^\circ \left( x_0;w \right)+r\left( \frac{f_i}{F_i} \right)^{\circ\circ} \left( x_0;v \right) <0,\, \exists\, i\in I,\\
	\left( w,r \right) \in \widetilde{T}^2\left( X,x_0,v \right).
\end{array} \right.
\end{equation}
Consequently, invoking Lemma \ref{yl:1} by the similar proof presented in \cite[Lemma 4.1]{Feng2019}, one can obtain that the system \eqref{xt:1.1} has no solution $(w,r)\in Z \times \mathbb{R}_+$. This completes the proof.
\end{proof}

We now derive the second-order strong KKT necessary conditions for a Borwein-type proper efficient solution of NMFP with the help of GSOARC.

\begin{theorem}\label{mt:1}{\rm [Primal condition]}
Let $x_0 \in U$ be a Borwein-properly efficient solution of {\rm NMFP} and $s=(s_1,s_2,\ldots,s_p)$ be defined as that in Lemma \ref{lem:2.1}. Let Assumption ~\ref{js:1} be fulfilled. If {\rm GSOARC} holds at $x_0$ in the direction $v\in D(x_0)$, then for any $r\ge0$, the  system
\begin{equation}\label{xt:2}
\left\{ \begin{array}{l}
f_{i}^{\circ}\left( x_0;w \right) +rf_{i}^{\circ\circ}\left( x_0;v \right)-s_i(\langle F'_{i,G}\left( x_0 \right), w \rangle + rF''_{i}( x_0;v )) \leq 0,\,\forall\, i\in I,\\
f_{i}^{\circ}\left( x_0;w \right) +rf_{i}^{\circ\circ}\left( x_0;v \right)-s_i(\langle F'_{i,G}\left( x_0 \right), w\rangle + r F''_{i}( x_0;v)) <0,\,\exists \, i\in I,\\
g_{j}^{\circ}\left( x_0 ;w\right) +rg_{j}^{\circ\circ}\left( x_0;v \right) \leq 0,\,\forall\, j\in J( x_0,v),\\
h_{k}^{\circ}\left( x_0 ;w\right) +rh_{k}^{\circ\circ}\left( x_0;v \right)=0,\,\forall\, k\in K,\\
\end{array} \right.
\end{equation}
is incompatible in $w \in Z$.
\end{theorem}

\begin{proof}
From Lemma \ref{yl:3}, it follows that there is no $(w,r)\in Z \times \mathbb{R}_+$ such that the system \eqref{xt:1} is consistent. Since GSOARC holds at $x_0$ in the direction $v$, combining \eqref{abd:1} with Definition \ref{def:7}, for all $r\ge0$, the system \eqref{xt:2} does not have a solution $w \in Z$. The proof is completed.
\end{proof}

\begin{theorem}\label{dl:1}{\rm [Dual condition]}
Let $x_0 \in U$ be a Borwein properly efficient solution of {\rm NMFP} and $s=(s_1,s_2,\ldots,s_p)$ be defined as stated in Lemma \ref{lem:2.1}.
Let Assumption ~\ref{js:1} be fulfilled and functions $f_i \,(i\in I)$, $g_j \,(j\in J)$ and $h_k \,(k\in K)$ be G\^ateaux differentiable at $x_0$.
Assume that {\rm GSOARC} holds at $x_0$ in the direction $v\in D(x_0)$. Then, there exist $\lambda \in \mathbb{R}^p_{++}$,
 $\mu \in \mathbb{R}^m_{+}$ and $\nu \in \mathbb{R}^l$ such that
\begin{align}
& \sum_{i=1}^p\lambda _i(f'_{i,G}(x_0)-s_iF'_{i,G}(x_0)) +\sum_{j=1}^m\mu _jg'_{j,G}(x_0)+\sum_{k=1}^l\nu _k h'_{j,G}(x_0) =0, \label{kkt:1} \\
& \sum_{i=1}^p\lambda _i(f_{i}^{\circ\circ}( x_0;v )-s_iF''_{i}( x_0;v )) +\sum_{j=1}^m\mu _jg_{j}^{\circ\circ}( x_0;v )+\sum_{k=1}^l\nu _kh_k^{\circ\circ}( x_0;v ) \geq 0, \label{kkt:2} \\
& \mu_jg_j(x_0)=0,\,\, j=1,2, \ldots, m \label{ys:1} \\
\text{ and }~ & \mu_j g'_{j,G}(x_0) v = 0,\,\, j=1,2, \ldots, m. \label{ys:2}
\end{align}
\end{theorem}

\begin{proof}
From Theorem \ref{mt:1}, it follows that there is no $(w,r) \in Z \times \mathbb{R}$ such that
\begin{equation}\label{xt:3}
\left. \begin{array}{l}
\langle f'_{i,G}( x_0 )-s_iF'_{i,G}( x_0 ), w \rangle + r(f_{i}^{\circ\circ}( x_0,v)-s_iF''_{i}( x_0,v)) \leq 0,\,\forall i\in I,\\
\langle f'_{i,G}( x_0 )-s_iF'_{i,G}( x_0 ), w\rangle + r(f_{i}^{\circ\circ}( x_0;v )-s_iF''_{i}( x_0;v )) <0,\,\exists i\in I,\\
\langle g'_{j,G}( x_0 ), w \rangle + r g_{j}^{\circ\circ}( x_0;v ) \leq 0,\,\forall j\in J( x_0,v),\\
\langle h'_{k,G}( x_0 ), w \rangle + r h_{k}^{\circ\circ}( x_0;v )=0,\,\forall k\in K,\\
-r\le 0.
\end{array} \right\}
\end{equation}
Let $J( x_0,v)=\{j_1,j_2,\ldots,j_t\}$. We consider the following matrices:
$$B:=\left( \begin{matrix}
f'_{1,G}( x_0 )-s_1F'_{1,G}( x_0 )	~~~~&~~~~\quad f_{1}^{\circ\circ}( x_0;v )-s_1F''_{1}( x_0;v )\\
f'_{2,G}( x_0 )-s_2F'_{2,G}( x_0 )	~~~~&~~~~\quad f_{2}^{\circ\circ}( x_0;v )-s_2F''_{2}( x_0;v )\\
	\vdots	& \vdots	\\
f'_{p,G}( x_0 )-s_pF'_{p,G}( x_0 )	~~~~&~~~~\quad f_{p}^{\circ\circ}( x_0;v )-s_pF''_{p}( x_0;v )\\
\end{matrix} \right),$$
$$C:=\left( \begin{matrix}
g'_{j_1,G}( x_0 )	\quad & \quad g_{j_1}^{\circ\circ}( x_0;v )\\
g'_{j_2,G}( x_0 )	\quad & \quad g_{j_2}^{\circ\circ}( x_0;v )\\
	\vdots	\quad & \quad \vdots	\\
g'_{j_t,G}( x_0 )	\quad & \quad g_{j_t}^{\circ\circ}( x_0;v )\\
0 & -1
\end{matrix} \right) ~~\text{ and }~~
D:=\left( \begin{matrix}
h'_{1,G}( x_0 )	\quad & \quad h_{1}^{\circ\circ}( x_0;v )\\
h'_{2,G}( x_0 )	\quad & \quad h_{2}^{\circ\circ}( x_0;v )\\
	\vdots	\quad & \quad \vdots	\\
h'_{l,G}( x_0 )	\quad & \quad h_{l}^{\circ\circ}( x_0;v )\\
\end{matrix} \right).$$
Then, note that the system \eqref{xt:3} is inconsistent in $(w,r)\in Z \times \mathbb{R}$ if and only if
$$\langle B, (w,r)\rangle \le 0,\,\, \langle C, (w,r)\rangle \leqq0,\,\, \langle D, (w,r)\rangle = 0$$
has no solution $(w,r) \in Z \times \mathbb{R}$.
This together with \cite[Theorem 3.24]{Ehrgott2005} yield that there exist $\lambda \in \mathbb{R}^p_{++}$, $\mu_j\ge0, j\in J(x_0,v)$, $\nu \in \mathbb{R}^l$ and $\mu_0\ge0$ such that
\begin{align*}
&  \sum_{i=1}^p{\lambda _i(f'_{i,G}( x_0 )-s_iF'_{i,G}( x_0 )) +\sum_{j=1}^m{\mu _jg'_{j,G}( x_0 )}}+\sum_{k=1}^l{\nu _k h'_{j,G}( x_0 ) =0}, \\
\text{ and }~ &  \sum_{i=1}^p\lambda _i(f_{i}^{\circ\circ}( x_0;v )-s_iF''_{i}( x_0;v ))+\sum_{j=1}^m\mu _jg_{j}^{\circ\circ}( x_0;v ) +\sum_{k=1}^l\nu _kh_k^{\circ\circ}( x_0;v ) =\mu_0\ge0.
\end{align*}
Consequently, we get \eqref{kkt:1}--\eqref{ys:2} by picking up $\mu_j=0$ when $j\notin J(x_0,v)$.  The proof is completed.
\end{proof}

\begin{remark}
Theorem \ref{dl:1} presents a second-order strong KKT necessary conditions, which contains some new information on multipliers \eqref{ys:2} at a Borwein properly efficient solution of the problem \eqref{mfp:1}. From the proof of Theorem \ref{dl:1}, we get a first-order strong KKT condition \eqref{kkt:1} at a Borwein properly efficient solution. As indicated in Remark \ref{rem:1}, when $v=0$, GSOARC reduces to \eqref{abd:3}. In this case, while \eqref{kkt:2} is trivial,  \eqref{kkt:1} and \eqref{ys:1} remain satisfied.
\end{remark}

\begin{remark}
\begin{enumerate}
\item[{\rm (i)}] If for each $i \in I, j\in J, k\in K$, $f_i$ and $g_{j}$ are continuously Fr\'echet differentiable at $x_0 \in U$ and $F_i(x) = 1$ and $h_{k}(x)\equiv 0$  for all $x \in U$,
and $f''_i(x_0, v)$ and $g''_j(x_0, v)$ exist for all $i \in I, j\in J(x_0, v)$, then
 Theorem 4.1 and  Theorem 4.2 in \cite{Feng2019} can be deduced by Theorem \ref{mt:1} and Theorem \ref{dl:1}, respectively.

\item[{\rm (ii)}] If for each $i \in I, j\in J, k\in K$, $f_i, g_{j}$ and $h_{k}$ are continuously Fr\'echet differentiable at $x_0 \in U$ and $F_i(x) = 1$  for all $x \in U$,
and $f''_i(x_0, v), h''_k(x_0, v)$ and $g''_j(x_0, v)$ exist for all $i \in I, j\in J(x_0, v)$, then  Theorem 4.1 and  Theorem 4.2 in \cite{Feng2018} can be recovered by Theorem \ref{mt:1} and Theorem \ref{dl:1}, respectively.
\end{enumerate}
\end{remark}

\begin{corollary}
Let $x_0$ be a Borwein properly efficient solution of {\rm NMFP } and $s$ $=$ $(s_1$, $s_2$, $\ldots$, $s_p)$ be defined as that in Lemma \ref{lem:2.1}. Let Assumption \ref{js:1} be fulfilled and functions $f_i \,(i\in I)$, $g_j \,(j\in J)$ and $h_k\,(k\in K)$ be G\^ateaux differentiable at $x_0$. If \eqref{abd:3} holds at $x_0$, then there exist vectors $\lambda \in \mathbb{R}^p_{++}$, $\mu \in \mathbb{R}^m_{+}$ and $\nu \in \mathbb{R}^l$ such that \eqref{kkt:1} and \eqref{ys:1} hold.
\end{corollary}

\begin{remark}
It is worth noting that $\lambda _i(f'_{i,G}( x_0 )-s_i F'_{i,G}( x_0 ))$ is identical to $\bar{\lambda} _i\left(\frac{f_i}{F_i}\right)'_G(x_0)$ by choosing $\bar{\lambda} _i=\frac{\lambda _i}{F_i(x_0)}$ according to Proposition \ref{pro:2.1}. Analogously, $\lambda _i(f_{i}^{\circ\circ}( x_0;v )-s_i F''_{i}( x_0;v))$ is identical to $\bar{\lambda} _i\left(\frac{f}{F} \right)^{\circ\circ} ( x_0;v )$.
\end{remark}

Since GSOGRC is weaker than GSOARC,  a natural question is whether the results presented in Theorem \ref{dl:1} hold under the GSOGRC assumption.
The answer is not affirmative, as evident from the following Example \ref{zjl:1}.

\begin{example}\label{zjl:1}
 Let $p=m=2$ and $l=1$.
Consider the  NMFP problem where the functions $f_i, F_i, g_j, h_k : \mathbb{R}^3\rightarrow \mathbb{R}$ are
as follows:
\begin{align*}
& f_1(x)=-3x_1+x_2,\, f_2(x)=2x_1-3x_2,\, F_1(x)=F_2(x)=1+x_1+x_2,  \\
& g_1(x)=-x_1,\, g_2(x)=-x_2 ~\text{ and }~ h(x)=x_1x_2.
\end{align*}
The feasible set for this problem is
$$X=\left\{ \left( x_1,x_2,x_3 \right) \in \mathbb{R}^3:x_1\ge 0,x_2\ge 0,x_1x_2=0 \right\}.$$
Clearly, $x_0=0$ is a Borwein properly efficient solution for NMFP  since
  $$T\left( \left(\frac{f}{F}\right)\left( X \right) +\mathbb{R}_{+}^{2},\left(\frac{f}{F}\right)\left( x_0 \right) \right) \cap \left( -\mathbb{R}_{+}^{2} \right) =\left\{ 0 \right\}.$$
By directly calculation, we  obtain $T(X,x_0)=X$, $s=(0,0)$,
\begin{align*}
& C(Q,x_0)=\left\{ \left( v_1,v_2,v_3 \right) \in \mathbb{R}^3:v_1\ge 0,v_1\le v_2\le 6v_1 \right\}, \\
& D(x_0)=T(X,x_0)\cap C(Q,x_0)=\left\{ \left( v_1,v_2,v_3 \right) \in \mathbb{R}^3:v_1=v_2=0 \right\},
\end{align*}
and for arbitrary $v\in D(x_0)$,
\begin{align*}
& \widetilde{T}^2\left( X,x_0,v \right)=\left\{ \left( w,s \right) \in \mathbb{R}^3\times \mathbb{R}:w\in T\left( X,x_0 \right), s\ge 0 \right\} \\
\text{ and }~ & \widetilde{C}^2\left( Q,x_0,v \right)=\left\{ \left( w,s \right) \in \mathbb{R}^3\times \mathbb{R}:w\in C\left( Q,x_0 \right), s\ge 0 \right\}.
\end{align*}
It can be seen that GSOARC \eqref{abd:1} is not satisfied and GSOGRC \eqref{gui:1} is satisfied at $x_0$.
We point out that the conclusion derived in Theorem \ref{dl:1} is not true for this example.
Indeed, because if there exist $\lambda_1>0$, $\lambda_2>0$, $\mu_1\ge0$, $\mu_2\ge0$ and $\nu\in \mathbb{R}$
 such that \eqref{kkt:1} is fulfilled, i.e.,
\begin{align*}
& \lambda _1(\nabla f_1\left( x_0 \right)-s_1\nabla F_1(x_0)) +\lambda _2(\nabla f_2\left( x_0 \right)-s_2\nabla F_2(x_0)) \\
& + \mu _1\nabla g_1\left( x_0 \right) + \mu_2 \nabla g_2\left( x_0 \right) +\nu \nabla h\left( x_0 \right) = 0,
\end{align*}
then $\lambda _1=-\frac{1}{7}(3\mu _1+2\mu _2) \le 0$ and $\lambda _2=-\frac{1}{7}(\mu _1+3\mu _2) \le 0$,
which are contradictory to $\lambda_1 > 0$ and $\lambda_2 > 0$.
Therefore, Theorem \ref{dl:1} is not true under the GSOGRC assumption.
\end{example}

\section{Second-order KKT sufficient optimality conditions for NMFP}\label{sec:4}
This section studies second-order sufficient optimality conditions for NMFP under second-order generalized convexity assumptions. To start with, we introduce some notions of second-order generalized convexity.

\begin{definition}\label{tx:1}
Let $x_0\in X\subseteq U$. A locally Lipschitz continuous function $\vartheta: U \rightarrow \mathbb{R}$ is said to be
\begin{enumerate}[(i)]
\item  second-order \emph{convex} at $x_0$ with respect to $X$ if for any $x\in X$ there exist $v\in Z\setminus \{0\}$ and $w\in Z \setminus\{0\}$ such that
$$\vartheta( x) -\vartheta \left( x_0 \right) \geq \vartheta^{\circ}\left( x_0;w\right) +\tfrac{1}{2}\vartheta^{\circ\circ}\left( x_0;v \right),$$

\item  second-order \emph{pseudoconvex} at $x_0$ with respect to $X$ if  for any $x\in X$ there exist $v\in Z \setminus\{0\}$ and $w\in Z \setminus\{0\}$ such that
   $$\vartheta^{\circ}\left( x_0;w\right) +\tfrac{1}{2}\vartheta^{\circ\circ}\left( x_0;v \right) \geq 0 ~\Longrightarrow~ \vartheta\left( x \right) \geq \vartheta\left( x_0 \right),$$

\item  second-order \emph{quasiconvex} at $x_0$ with respect to $X$ if  for any $x\in X$ there exist $v\in Z \setminus\{0\}$ and $w\in Z \setminus \{0\}$ such that
   $$\vartheta\left( x \right) \leq \vartheta\left( x_0 \right) ~\Longrightarrow~ \vartheta^{\circ}\left( x_0;w\right) +\tfrac{1}{2}\vartheta^{\circ\circ}\left( x_0;v \right) \leq 0,$$
\item  second-order \emph{infine} at $x_0$ with respect to $X$ if for any $x\in X$ there exist $v\in Z\setminus\{0\}$ and $w\in Z\setminus\{0\}$ such that
   $$\vartheta\left( x \right) -\vartheta\left( x_0 \right) = \vartheta^{\circ}\left( x_0;w\right) +\tfrac{1}{2}\vartheta^{\circ\circ}\left( x_0;v \right).$$
\end{enumerate}
\end{definition}

\begin{remark}\label{rem:2}
 It is noteworthy that if a function $\vartheta$ is second-order infine at $x_0$ with respect to $X$, then  $\vartheta$ is second-order convex at $x_0$ with respect to $X$. The second-order convexity of $\vartheta$  at $x_0$ with respect to $X$ implies the second-order pseudoconvexity and second-order quasiconvexity of $\vartheta$ at $x_0$ with respect to $X$. Besides, if $\vartheta: U \rightarrow \mathbb{R}$ is locally Lipschitz continuous and  G\^ateaux differentiable at $x_0$, and $w=v= x- x_0$ in Definition \ref{tx:1}(i)(ii)(iii), then the second-order convexity, second-order pseudoconvexity and second-order quasiconvexity of $\vartheta$  at $x_0$ with respect to $X$ reduce to the corresponding second-order convexity of $\vartheta$  at $x_0$ introduced in \cite[Definition 4.1]{Luu}.
\end{remark}

\begin{example}
\begin{enumerate}
\item[{\rm (i)}]
Consider the function $\vartheta(x) = x^2$, $x \in \mathbb{R}$, and $x_0=0$. By direct calculation, we get $\vartheta^{\circ}( 0;w ) =\nabla \vartheta( 0) w =0$ for all $w\in \mathbb{R}$, and $\vartheta^{\circ\circ}(0;v)=\nabla ^2\vartheta( 0 ) v^2=2v^2\ge0$. Letting $v=x-x_0=x$, one has
$$\vartheta( x ) -\vartheta( x_0) = \vartheta^{\circ}( x_0;w) + \tfrac{1}{2}\vartheta^{\circ\circ}( x_0; v ).$$
Therefore, $\vartheta$ is second-order infine at $x_0$ with respect to  $\mathbb{R}$.
In addition, from Remark \ref{rem:2}, it is obvious that $\vartheta$ is also second-order convex
(pseudoconvex and quasiconvex) at $x_0$ with respect to  $\mathbb{R}$.

\item[{\rm (ii)}]
Consider the function $\vartheta(x) = |x|$ and  $x_0 = 0$.
By direct calculation, we get $\vartheta^{\circ}( 0;w ) =|w|$,
and $\vartheta^{\circ\circ}(0;v)=0$ for all $v\in \mathbb{R}$. Hence, for any $x\ne0$,
letting $w:=x-x_0=x$, we get
    $\vartheta^{\circ}( x_0;w )+\tfrac{1}{2}\vartheta^{\circ\circ}( x_0; v )=|x|>0$ for all $v\in \mathbb{R}$.
So, $\vartheta$ is second-order pseudoconvex at $x_0$ with respect to $\mathbb{R}$.
\end{enumerate}
\end{example}

\begin{definition}\label{def:4.2}
A function $f:=\left( f_1,f_2,\ldots,f_p \right )^{\top}: U\rightarrow \mathbb{R}^p$ is called second-order convex (respectively, pseudoconvex, quasiconvex, and infine) at $x_0$ with respect to $X$ if its components $f_i, \,i\in I$ are second-order convex (respectively, pseudoconvex, quasiconvex, and infine) at $x_0$ with respect to $X$ and common $v \in Z $ and $w \in Z$.
\end{definition}

By direct calculation, one easily gets the following results.

\begin{proposition}\label{second:convex}
Let $f_{1}$ and $f_{2}$ be two real-valued locally Lipschitz functions on $U$ with $f_{2}$ being positive, and $x_0\in X\subseteq U$.
 Suppose that for each $x\in U$, both $f'_{2,G}(x)$ and $f''_{2}(x;v)$ exist. Then,
\begin{enumerate}
\item[{\rm (i)}]
 for each $\beta>0$, $\beta f_{1}$ is second-order convex (pseudoconvex, quasiconvex, or infine) at $x_0$ with respect to $X$ if so is $f_{1}$;

\item[{\rm (ii)}]
$f_{1}+f_{2}$ is second-order convex (infine) at $x_0$ with respect to $X$ if so are $f_{1}$ and $f_{2}$ with a common $v\in Z$ and $w\in Z$.
\end{enumerate}
\end{proposition}

For simplicity, we take the following assumption to derive second-order sufficient optimality conditions.
\begin{framed}
\begin{assumption}\label{js:2}
All the functions involved in {\rm NMFP} exhibit a form of
second-order convexity in the sense of Definition \ref{def:4.2}
with a common $v\in D(x_0)$ and $w\in Z$.
\end{assumption}
\end{framed}

Next, we present second-order sufficient optimality conditions for NMFP.

\begin{theorem}\label{dl:2}
Let $x_0\in X$, $s=\frac{f(x_0)}{F(x_0)}$ and let $f-s\ast F$ and $g$ be second-order convex at $x_0$ with respect to $X$, $h$ be second-order infine at $x_0$ with respect to $X$ and Assumptions  \ref{js:1} and \ref{js:2} hold.
Assume that there exist $\lambda \in \mathbb{R}^p_{++}$, $\mu \in \mathbb{R}^m_+ $ and $ \nu \in \mathbb{R}^l$  such that for all $v \in D(x_0)$ and $w \in Z$,
\begin{align}
& \lambda^\top (f^{\circ}( x_0;w ) - \langle s*F'_{G}( x_0 ), w\rangle ) + \sum_{j=1}^m \mu_jg_{j}^{\circ}( x_0;w ) + \sum_{k=1}^l \nu _kh_{k}^{\circ}( x_0;w ) = 0, \label{cf:3}  \\
& \lambda^\top (f^{\circ\circ}( x_0; v ) - s * F''(x_0;v)) + \sum_{j=1}^m{\mu_j g_{j}^{\circ\circ}( x_0;v )} + \sum_{k=1}^l{\nu _kh_{k}^{\circ\circ}( x_0;v )}\ge 0 \label{cf:2} \\
\text{ and }~ & \sum_{j=1}^m \mu_jg_j(x_0) = 0. \label{cf:1}
\end{align}
Then, $x_0$ is a Pareto efficient solution of {\rm NMFP}.
\end{theorem}

\begin{proof}
From Lemma \ref{lem:2.1}, it suffices to prove that $x_0$ is a Pareto efficient solution of the problem \eqref{smfp:1}.
On the contrary, if $x_0$ is not a Pareto efficient solution of the problem \eqref{smfp:1}, then there exists $\bar{x}\in X$ for which
\begin{align}\label{eq24}
(f-s\ast F)(\bar x)\le (f-s\ast F)(x_0).
\end{align}
Due to $\bar{x},x_0\in X$ and $\lambda \in \mathbb{R}^p_{++}$, \eqref{cf:1} and \eqref{eq24} yield that
\begin{equation}\label{fz:1}
\begin{aligned}
&\lambda^\top (f - s*F)\left( x_0 \right) + \sum_{j=1}^m{\mu _jg_j\left( x_0 \right)}+\sum_{k=1}^l{\nu _kh_k\left( x_0 \right)} \\
& > \lambda^\top (f - s*F)\left( \bar{x} \right) +\sum_{j=1}^m{\mu _jg_j\left( \bar{x} \right)}+\sum_{k=1}^l{\nu _kh_k\left( \bar{x} \right)}.
\end{aligned}
\end{equation}
Note that $f-s\ast F$ and $g$ are second-order convex at $x_0$ with respect to $X$, and $h$ is second-order infine at $x_0$ with respect to $X$.
Using Proposition \ref{pro:2.1}, Assumption \ref{js:2} yields that there exist $v\in D(x_0)$ and $ w \in Z $ such that
\begin{align}\label{star_dg}
~&~  \lambda^\top (f - s*F)( \bar{x} ) + \sum_{j=1}^m{\mu _jg_j( \bar{x} )}+\sum_{k=1}^l{\nu _kh_k( \bar{x} )} \notag \\
 \geq ~&~  \lambda^\top (f - s*F)( x_0 ) +\sum_{j=1}^m{\mu _jg_j( x_0 )} +\sum_{k=1}^l{\nu _kh_k( x_0 )}  \notag \\
~&~ + \lambda^\top (f - s*F)^{\circ}( x_0;w ) + \sum_{j=1}^m{\mu _jg_{j}^{\circ}( x_0;w ) }+\sum_{k=1}^l{\nu _kh_{k}^{\circ}( x_0;w ) }  \notag \\
~&~ + \tfrac{1}{2} \lambda^\top (f - s*F)^{\circ \circ}( x_0;v ) + \sum_{j=1}^m{\tfrac{1}{2}\mu _jg_{j}^{\circ \circ}\left( x_0;v \right)}+\sum_{k=1}^l{\tfrac{1}{2}\nu _kh_{k}^{\circ \circ}( x_0;v )}  \notag \\
= ~&~ \lambda^\top (f - s*F)( x_0 ) +\sum_{j=1}^m{\mu _jg_j( x_0 )} + \sum_{k=1}^l{\nu _kh_k( x_0 )}  \notag \\
~&~ + \lambda^\top (f^{\circ}( x_0; w ) - \langle s*F'_{G}( x_0 ), w\rangle) + \sum_{j=1}^m{\mu _jg_{j}^{\circ}( x_0;w )}+\sum_{k=1}^l{
\nu _kh_{k}^{\circ}( x_0;w )}  \notag \\
~&~ + \tfrac{1}{2}\left(\lambda^\top (f^{\circ\circ}\left( x_0;v \right)-s*F''(x_0;v)) + \sum_{j=1}^m{\mu _jg_{j}^{\circ\circ}\left( x_0;v \right)}+\sum_{k=1}^l{\nu _kh_{k}^{\circ\circ}( x_0;v )}\right)  \notag \\
\overset{\eqref{cf:3} \& \eqref{cf:2}}{\ge} ~&~  \lambda^\top (f - s*F) ( x_0 ) + \sum_{j=1}^m{\mu _jg_j ( x_0 )} + \sum_{k=1}^l{\nu _k h_k ( x_0 )},
\end{align}
which is contradictory to \eqref{fz:1}.
Therefore, $x_0$ is a Pareto efficient solution of the problem \eqref{smfp:1}, and hence the proof is completed.
\end{proof}

\begin{remark}
From the proof of Theorem \ref{dl:2}, we see that if $\lambda \in \mathbb{R}^p_{++}$ is replaced by $\lambda \in \mathbb{R}^p_{+}\setminus\{0\}$ and the inequality ($\ge$) in \eqref{cf:2} is replaced by the strict inequality ($>$), then $x_0$ is a weak efficient solution of NMFP under the assumptions of Theorem \ref{dl:2}.
\end{remark}

\begin{theorem}\label{dl:3}
Let $x_0\in X$ and $s=\frac{f(x_0)}{F(x_0)}$.
Assume that there exist $\lambda \in \mathbb{R}^p_{++}$, $\mu \in \mathbb{R}^m_+ $ and $\nu \in \mathbb{R}^l$ such that \eqref{cf:3}--\eqref{cf:1} hold.
If $\lambda ^\top  (f-s\ast F)$ is second-order pseudoconvex at $x_0$ with respect to $X$,
$\mu ^\top g$ is second-order quasiconvex at $x_0$ with respect to $X$,
 $h$ is second-order infine at $x_0$ with respect to $X$ and Assumptions  \ref{js:1} and \ref{js:2} hold,
 then $x_0$ is a Pareto efficient solution of {\rm NMFP}.
\end{theorem}

\begin{proof}
Consider the functions $\phi( x ) :=\lambda^{\top} (f-s\ast F)(x)$ and $\psi( x ) :=\mu^{\top} g( x )$ for $x\in X$. Then, we have
$$\psi(x)\leq \psi(x_0)\le 0\,\, \text{ and }~   h(x)=h(x_0)=0,\,\,\forall\,  x\in X.$$
Since $\mu ^\top g$ is second-order quasiconvex at $x_0$ with respect to $X$ and $h$ is second-order infine at $x_0$ with respect to $X$, Assumptions  \ref{js:2} yields that for any $x\in X$, there exist
 $v\in D(x_0)$ and $w\in Z$ such that
\begin{align*}
\psi^{\circ}( x_0;w ) +\tfrac{1}{2}\psi^{\circ\circ}( x_0;v ) \leq 0,\,\,\,
h_k^{\circ}( x_0;w )+\tfrac{1}{2}h_k^{\circ\circ}( x_0;v )= 0,\,\, k\in K.
\end{align*}
Moreover, we have
\begin{align*}
\sum_{j=1}^m{\mu _jg_{j}^{\circ}( x_0;w ) +\tfrac{1}{2}\sum_{j=1}^m{\mu _jg_{j}^{\circ\circ}( x_0;v )}}\leq 0
~\text{ and }~
\sum_{k=1}^l{\nu _k h_k^{\circ}( x_0;w )+\tfrac{1}{2}\sum_{k=1}^l{\nu _kh_k^{\circ\circ}( x_0;v )}}=0.
\end{align*}
From \eqref{cf:3} and \eqref{cf:2}, we deduce that
\begin{align*}
\lambda^\top (f^{\circ}( x_0;w ) - \langle s*F'_{G}( x_0 ), w\rangle) +\tfrac{1}{2} \lambda^\top (f^{\circ\circ}( x_0;v )- s*F''( x_0;v )) \ge 0,
\end{align*}
i.e., $\phi^{\circ}( x_0;w ) +\tfrac{1}{2}\phi ^{\circ\circ}( x_0;v ) \ge 0.$
Since $\phi ( x )$ is second-order pseudoconvex  at $x_0$ with respect to $X$, we have $\phi ( x )\geq \phi( x_0 )$ for all $x\in X$, and so
\begin{equation}\label{yx:1}
\lambda^\top (f - s*F)( x ) \geq \lambda^\top (f - s*F)(x_0),\,\, \forall\, x\in X.
\end{equation}
Due to $\lambda \in \mathbb{R}^p_{++}$, \eqref{yx:1} implies that
\begin{align*}
  (f-s*F)(x_0)-(f-s*F)(x)\notin \mathbb{R}^p_{+} \setminus \{0\},\,\forall \,x\in X.
 \end{align*}
It, therefore, follows from Definition \ref{def:6} that $x_0$ is a Pareto efficient solution of the problem \eqref{smfp:1}, and hence it is a Pareto efficient solution of NMFP by Lemma \ref{lem:2.1}.
 The proof is completed.
\end{proof}

In the following, we exemplify Theorem \ref{dl:3}.
\begin{example}
Let $p=m=2$ and $l=1$.
Consider the NMFP problem, where the functions $f_i$, $F_i$, $g_j$, $h_k$: $\mathbb{R}^2\rightarrow \mathbb{R}$
are defined by
\begin{align*}
& f_1(x):=3x_1^4+5x_1^2+6x_2^2,\,\, f_2(x):=-2x_2^2,\,\, F_1(x)=F_2(x):=x_1^2+x_2^2+1, \\
& g_1(x):=-x_1^2,\, g_2(x):=-x_2 \text{ and } h(x):=x_1^2.
\end{align*}
The feasible set for the problem is
$X=\left\{ ( x_1,x_2) \in \mathbb{R}^2:x_1= 0,x_2\ge 0 \right\}$, and for $x_0=(0,0)$, we get $s=(0,0)$ and $D(x_0)=X$. Take $\lambda_1=1$, $\lambda_2=2$, $\mu_1=1$, $\mu_2=0$ and $\nu=-2$. For each $x=(x_1,x_2) \in X$, there exist $v=x\in D(x_0)$ and $w=(1,1)$ such that $\lambda ^\top (f-s\ast F)$ is
 second-order convex (and pseudoconvex) at $x_0$ with respect to $X$,
 $\mu ^\top g$ is second-order quasiconvex at $x_0$ with respect to $X$, and $h$ is second-order infine at $x_0$ with respect to $X$.
In addition, for any $v\in D(x_0)$, we have \eqref{cf:3}--\eqref{cf:1}.  Observe that for any $x\in X$,
\begin{align*}
\frac{f(x_0)}{F(x_0)}- \frac{f(x)}{F(x)} = \left(  -  \frac{3x_1^4+5x_1^2+6x_2^2}{x_1^2+x_2^2+1}, \frac{2x_2^2}{x_1^2+x_2^2+1}   \right)^{\top} =  \left(  -  \frac{6x_2^2}{x_2^2+1}, \frac{2x_2^2}{x_2^2+1}   \right)^{\top} \notin \mathbb{R}^2_{+}\setminus \{0\}.
\end{align*}
It thus implies that $x_0$ is a Pareto efficient solution of NMFP.
\end{example}

\begin{remark}
From the proof of Theorem \ref{dl:3}, specifically beginning with \eqref{yx:1}, we can conclude that if $\lambda \in \mathbb{R}^p_{++}$ is replaced by $\lambda \in \mathbb{R}^p_{+} \setminus \{0\}$ in Theorem \ref{dl:3}, then $x_0$ is a weak efficient solution of NMFP under the assumptions of Theorem \ref{dl:3}.
\end{remark}



\section{Second-order duality}\label{sec:5}
In this section, we study  Mond-Weir-type second-order duality of NMFP.  Duality results between NMFP and its second-order dual problem are established under the generalized second-order convexity assumptions. A  \emph{Mond-Weir-type second-order  dual problem} (MWSD) of NMFP is formulated as follows:
\begin{equation}\label{mwd:1}
\resizebox{\textwidth}{!}{$
\begin{aligned}
\max \, ~&~~ \left(\frac{f}{F}\right)( u)\\
\mbox{s.t.}\, ~&~~ \lambda^\top(f'_{G}(u) - s*F'_{G}(u)) +\sum_{j=1}^m\mu _jg'_{j,G}( u ) +\sum_{k=1}^l\nu _k h'_{k,G}( u)=0,\\
~&~~ \lambda^\top(f^{\circ \circ}( u;v ) - s*F''( u;v )) + \sum_{j=1}^m\mu _jg_{j}^{\circ \circ}( u;v ) + \sum_{k=1}^l\nu _kh_k^{\circ \circ}( u ;v) \geq 0,\,\,\forall\, v\in D(u),\\
~&~~\sum_{j=1}^m \mu _jg_j( u)+\sum_{k=1}^l\nu _kh_k( u)\geq 0,\quad\sum_{i=1}^p \lambda_i=1,\\
~&~~(u,\lambda,\mu,\nu)\in U\times \mathbb{R}^p_{++} \times \mathbb{R}^m_{+}\times \mathbb{R}^l,
\end{aligned}
$}
\end{equation}
where $s=\frac{f(u)}{F(u)}$. A vector $(u,\lambda,\mu,\nu)$  satisfying all the constraints of MWSD is said to be a feasible solution of MWSD.
The set of feasible solutions of MWSD is denoted by  $\mathcal{F}_M$.

Throughout this section,  we always assume Assumption \ref{ass:5.1} hold.
\begin{framed}
\begin{assumption}\label{ass:5.1}
\begin{enumerate}
  \item[{\rm (i)}]
  $f_i, F_i\,(i\in I),\, g_j\, (j\in J)$ and $h_k\, (k\in K)$ are all locally Lipschitz continuous, G\^ateaux differentiable and regular in the sense of Clarke on $U$ with G\^ateaux derivative $f'_{i,G},\, F'_{i,G}, \, g'_{j,G}$ and $h'_{k,G}$, respectively;

  \item[{\rm (ii)}] For $i\in I,j\in J,k\in K$,
   $f_i^{\circ\circ}(u;v),\, g_j^{\circ\circ}(u;v),\, h_k^{\circ\circ}(u;v)$ and $F''_i(u,v)$ are all finite at each $u\in U$ for all  directions $v\in Z$.
\end{enumerate}
\end{assumption}
\end{framed}

We first give the weak duality between  NMFP and MWSD.

\begin{theorem}\label{dl:7} {\rm [Weak duality]}
Let $x\in X$ and $(u,\lambda,\mu,\nu)\in \mathcal{F}_M$.
Assume that $f-s\ast F$ and $g$ are second-order convex  at $u$ with respect to $U$,
 $h$ is second-order infine at $u$ with respect to $U$ and Assumption \ref{js:2} holds. Then,
\begin{align*}
  \left(\frac{f}{F}\right)( x ) \nleq \left(\frac{f}{F}\right) ( u ).
\end{align*}
\end{theorem}

\begin{proof}
Suppose to the contrary that
\begin{equation}\label{mwd:1:1}
 \left(\frac{f}{F}\right)( x ) \le \left(\frac{f}{F}\right) ( u ).
\end{equation}
Taking into account that $F(x)\in \mathbb{R}_{++}^{p}$ for all $x\in X$, \eqref{mwd:1:1} implies that
\begin{align*}
\left(\frac{f}{F}\right)( x ) -\left(\frac{f}{F}\right) ( u ) \in - \mathbb{R}_{+}^{p} \setminus \{0\}
  &\Longleftrightarrow f(x)-s\ast F(x) \in - \mathbb{R}_{+}^{p} \setminus \{0\}\\
  &\Longleftrightarrow f(x)-s\ast F(x)\le f(u)-s\ast F(u).
\end{align*}
This together with $\lambda\in \mathbb{R}_{++}^{p}$ yields that
\begin{align}\label{mwd:2}
\lambda^\top (f - s*F)(x) < \lambda^\top (f - s*F)(u).
\end{align}
Since $x\in X$ and $(u,\lambda,\mu,\nu)\in \mathcal{F}_M$, we have  $g(x)\leqq0$, $h(x)=0$ and
$$\sum_{j=1}^m{\mu _jg_j( u )}+\sum_{k=1}^l\nu _k h_k ( u )\geq 0.$$
Combining with \eqref{mwd:2}, we have
\begin{equation}\label{eqn30:dg}
\begin{aligned}
&\lambda^\top (f - s*F)(x) + \sum_{j=1}^m{\mu _jg_j( x )} + \sum_{k=1}^l{\nu _kh_k( x )} \\
&< \lambda^\top (f - s*F)(u) + \sum_{j=1}^m{\mu _jg_j( u )} + \sum_{k=1}^l{\nu _kh_k( u )}.
\end{aligned}
\end{equation}
By the similar arguments as that in deriving \eqref{star_dg} in Theorem \ref{dl:2}, we get
\begin{equation*}
\begin{aligned}
& \sum_{i=1}^p \lambda_i(f_i-s_iF_i)(x) +\sum_{j=1}^m \mu_jg_j( x ) +\sum_{k=1}^l\nu_kh_k( x )\\
 & \geq\sum_{i=1}^p \lambda_i(f_i-s_iF_i)(u)+\sum_{j=1}^m \mu_jg_j( u )+\sum_{k=1}^l \nu_kh_k( u ),
\end{aligned}
\end{equation*}
which is a contradictory to \eqref{eqn30:dg}. Hence, the result follows. The proof is completed.
\end{proof}

\begin{theorem}\label{dl:8}{\rm [Strong duality]}
Let $x_0\in X$ be a Borwein-properly efficient solution of {\rm NMFP}.
Assume that all conditions of Theorem \ref{dl:1}  are satisfied.
Then, there exist $\bar{\lambda} \in \mathbb{R}^p_{++}$, $\bar{\mu} \in \mathbb{R}^m_{+}$ and $\bar{\nu} \in \mathbb{R}^l$ such that $(x_0,\bar{\lambda},\bar{\mu},\bar{\nu})\in \mathcal{F}_M$.
Furthermore, if all conditions of Theorem \ref{dl:7} hold, then $(x_0,\bar{\lambda},\bar{\mu},\bar{\nu})$ is a Pareto efficient solution of {\rm MWSD}.
\end{theorem}

\begin{proof}
From Theorem \ref{dl:1} it follows that there exists $(\bar{\lambda},\bar{\mu},\bar{\nu})\in \mathbb{R}^p_{++}\times \mathbb{R}^m_{+}\times \mathbb{R}^l$ satisfies \eqref{kkt:1}-\eqref{ys:1}.
Due to $h(x_0)=0$, \eqref{ys:1} yields that
\begin{align*}
\sum_{j=1}^m \bar{\mu}_jg_j( x_0 ) +\sum_{k=1}^l \bar{\nu}_kh_k( x_0 )=0.
\end{align*}
With no loss of generality, we can assume that $\sum_{i=1}^p{\bar{\lambda}_i}=1$
since one can pick up
$\widetilde{\lambda}_i=\frac{\bar{\lambda}_i}{\sum_{i=1}^p{\bar{\lambda}_i}}$ due to $\bar{\lambda}\in \mathbb{R}^p_{++}$.
Consequently, one has $(x_0,\bar{\lambda},\bar{\mu},\bar{\nu})\in \mathcal{F}_M$.
By Theorem \ref{dl:7}, we obtain
\begin{align*}
  \left(\frac{f}{F}\right)(x_0) \nleq \left(\frac{f}{F}\right) ( u ),\,\, \forall\, (u,\lambda,\mu, \nu)\in \mathcal{F}_M.
\end{align*}
It, therefore, implies that $(x_0,\lambda,\mu,\nu)$ is a Pareto efficient solution of MWSD.
The proof is completed.
\end{proof}

\begin{theorem}\label{dl:9}{\rm [Converse duality]}
Let $(u, \lambda, \mu, \nu)\in \mathcal{F}_M$ be a Pareto efficient solution of {\rm MWSD} with $u\in X$. Assume that $\lambda ^{\top }(f-s\ast F)$ is second-order pseudoconvex  at $u$ with respect to $U$, $\mu ^{\top }g$ is second-order quasiconvex  at $u$ with respect to $U$, $h$ is second-order infine  at $u$ with respect to $U$ and Assumption \ref{js:2} holds. Then $u$ is a Pareto efficient solution of {\rm NMFP}.
\end{theorem}

\begin{proof}
   If possible, suppose that $u$ is not a Pareto efficient solution of NMFP. Then,  there exists $\widetilde{x}\in X$ for which
\begin{align*}
  \left(\frac{f}{F}\right)(\widetilde{x}) \leq \left(\frac{f}{F}\right) ( u ).
\end{align*}
Therefore, $f(\widetilde{x})-s\ast F(\widetilde{x})\leq 0= f ( u )-s\ast F ( u )$, and
\begin{align}\label{converse:2}
 \lambda^{\top } \left(f(\widetilde{x})-s\ast F(\widetilde{x}) \right) < \lambda^{\top } \left( f ( u )-s\ast F ( u )\right).
\end{align}
Due to $(u, \lambda, \mu, \nu)\in \mathcal{F}_M$, \eqref{mwd:1} results in
\begin{equation}\label{mcd:1}
\begin{split}
 &\sum_{i=1}^p{\lambda_i(\langle f'_{i,G} \left( u \right), w\rangle - s_i \langle F'_{i,G}\left( u \right), w\rangle ) +\sum_{j=1}^m{\mu_j \langle g'_{j,G}\left( u \right), w\rangle}}\\
 &+\sum_{k=1}^l{\nu_k \langle h'_{k,G}\left( u \right), w\rangle }+\tfrac{1}{2}\sum_{i=1}^p{\lambda_i(f_{i}^{\circ \circ}\left( u;v \right)-s_iF''_i\left( u;v \right)) }\\ &+\tfrac{1}{2}\sum_{j=1}^m{\mu_jg_{j}^{\circ \circ}\left( u;v \right)}+\tfrac{1}{2}\sum_{k=1}^l{\nu_kh_k^{\circ \circ}\left(u ;v\right) }\geq 0,\,
 \forall\, v\in D(u), w \in Z,
\end{split}
\end{equation}
and $\mu^{\top }g(u)+\nu^{\top }h(u)\ge 0$. This together with $u\in X$ yields that
\begin{align}\label{converse:1}
\mu^\top g(u)=0\ge \mu^\top g(x),\,\,  h(u)=h(x)=0,\, \forall \, x\in X.
\end{align}
Since $\mu^{\top }g$ is second-order quasiconvex
 at $u$ with respect to $U$, $h$ is second-order infine  at $u$ with respect to $U$,
 it follows from Proposition \ref{second:convex} and Assumption \ref{js:2}  that   for any $x\in U$,
 there exist $v\in D(u)$ and $ w \in Z$ such that
\begin{align}
~&~ \sum_{j=1}^m{\mu _j \langle g'_{j,G}\left( u \right), w\rangle}+\tfrac{1}{2}\sum_{j=1}^m{\mu_jg_{j}^{\circ \circ}\left( u;v \right)}\le 0 \label{mcd:2} \\
\text{ and } ~&~ \sum_{k=1}^l{\nu _k \langle h'_{k,G}\left( u \right), w \rangle} + \tfrac{1}{2}\sum_{k=1}^l{\nu _kh_{k}^{\circ \circ}\left( u;v \right)}=\nu^\top h(x)-\nu^\top h(u)=0.\label{mcd:3}
\end{align}
Summing \eqref{mcd:1}, \eqref{mcd:2} and \eqref{mcd:3}, it yields that
\begin{equation*}
\sum_{i=1}^p \lambda _i( \langle f'_{i,G}\left( u \right) , w\rangle -s_i \langle F'_{i,G}\left( u \right), w\rangle ) + \tfrac{1}{2}\sum_{i=1}^p\lambda _i(f_{i}^{\circ \circ}\left( u;v \right)-s_iF''_i\left( u;v \right)) \ge 0.
\end{equation*}
By  the second-order pseudoconvexity of $\lambda^{\top }( f-s\ast F)$ at $u$ with respect to $U$,
 one has
\begin{equation*}
  \lambda ^{\top }(f-s\ast F)(x)\geq\lambda ^{\top }(f-s\ast F)(u),\,\,\forall\, x\in U\setminus\{u\},
\end{equation*}
which contradicts \eqref{converse:2}. Hence, the result follows. The proof is completed.
\end{proof}

\begin{remark}
If $\lambda \in \mathbb{R}^p_{++}$ is replaced by $\lambda \in \mathbb{R}^p_{+} \setminus \{0\}$ in MWSD, then we can conclude  the  strong duality and converse duality between the weak efficient solutions of NMFP and  that of MWSD.
\end{remark}

\section{Conclusions}\label{sec:6}

We have presented chain rules of quotient functions involving locally Lipschitz functions in terms of first and second-order directional derivatives, which improves that of \cite{VanSu2023,Feng2019,Feng2018}. A new second-order Abadie-type regular condition has been introduced in terms of Clarke directional derivative and P\'ales-Zeidan second-order directional derivative, which is different from the second-order Abadie-type regular conditions in \cite{Feng2019,Aanchal2023}. Second-order strong KKT conditions for a Borwein-properly efficient solution of NMFP have been established. Based on $s$-MFP,  second-order sufficient optimality conditions for a Pareto efficient solution of NMFP have been obtained under some generalized second-order convexity assumptions. Finally, we have proposed a Mond-Weir-type second-order dual problem of NMFP and obtained the weak, strong and converse duality results between NMFP and its second-order dual problem.

In the lines of the derived results, one can study Schaible-type second-order dual problem \cite{Liang2003} of NMFP, and also can attempt to design algorithms of NMFP via $s$-MFP. The obtained results in this paper can further be extended by using  Dini Hadamard-type second-order generalized directional derivatives.
It is also interesting to extend the results presented in this paper to the vector case when $\mathbb{R}^p_{+}$ is replaced by a general closed convex cone.
\vskip5mm

\noindent{\bf Acknowledgements.}
 This paper was supported by  the
   Natural Science Foundation of China (Nos. 12071379, 12271061),
 the Natural Science Foundation of Chongqing(cstc2021jcyj-msxmX0925,  cstc2022ycjh-bgzxm0097), Youth Project of Science and Technology Research Program of Chongqing Education Commission of China (No. KJQN202201802) and the Southwest University Graduate Research Innovation Program (No. SWUS23058). D. Ghosh is thankful for the research funding CRG (CRG/2022/001347) and MATRICS (MTR/2021/000696), SERB, India.
   J.-C. Yao acknowledges MOST (No. 111-2115-M-039-001-MY2), Taiwan.  The first version of the paper was finished on 10th May, 2023.
\vskip2mm

\noindent {\bf Author contributions.}
All authors contributed equally to this article.
\vskip2mm

\noindent {\bf Compliance with Ethical Standards}
 It is not applicable.

\vskip2mm
\noindent{\bf Competing interests}\rm

No potential conflict of interest was reported by the authors.

\end{document}